\documentclass{amsart}
\usepackage[T1]{fontenc}
\usepackage{amsmath}
\usepackage{amssymb}
\usepackage{amsthm}
\usepackage{extarrows}
\usepackage{enumerate}
\usepackage[hyphens]{url}
\usepackage{hyperref}
\usepackage{tikz}
\usepackage{tikz-cd}
\usetikzlibrary{decorations.pathmorphing}
\hypersetup{
    colorlinks=true,
    linkcolor=blue,
}

\usepackage[
    backend=biber,
    style=alphabetic,
    sorting=nyt,
    useprefix=true,
    doi=false,
    maxnames=99
]{biblatex}
\newtheorem{theorem}{Theorem}[section]
\newtheorem{proposition}[theorem]{Proposition}
\newtheorem{lemma}[theorem]{Lemma}
\newtheorem{corollary}[theorem]{Corollary}

\theoremstyle{definition}
\newtheorem{definition}[theorem]{Definition}
\newtheorem{example}[theorem]{Example}
\newtheorem{construction}[theorem]{Construction}
\newtheorem{remark}[theorem]{Remark}

\newcommand{\EE}{\mathbb{E}}

\newcommand{\NN}{\mathbb{N}}

\newcommand{\ZZ}{\mathbb{Z}}

\newcommand{\bA}{\mathbf{A}}
\newcommand{\bB}{\mathbf{B}}

\newcommand{\bP}{\mathbf{P}}
\newcommand{\bV}{\mathbf{V}}
\newcommand{\Ga}{\mathbf{G}_{\mathrm{a}}}

\newcommand{\Gash}{\hat{\mathbf{G}}_{\mathrm{a}}^\sharp}
\newcommand{\Gm}{\mathbf{G}_{\mathrm{m}}}

\newcommand{\cA}{\mathcal{A}}

\newcommand{\cC}{\mathcal{C}}
\newcommand{\cD}{\mathcal{D}}
\newcommand{\cE}{\mathcal{E}}
\newcommand{\cF}{\mathcal{F}}
\newcommand{\cG}{\mathcal{G}}
\newcommand{\cI}{\mathcal{I}}

\newcommand{\cL}{\mathcal{L}}
\newcommand{\M}{\mathcal{M}}
\newcommand{\N}{\mathcal{N}}
\newcommand{\cO}{\mathcal{O}}
\newcommand{\cP}{\mathcal{P}}

\newcommand{\Alg}{\mathsf{Alg}}
\newcommand{\Ani}{\mathsf{Ani}}

\newcommand{\CAlg}{\mathsf{CAlg}}

\newcommand{\Cat}{\mathsf{Cat}}

\newcommand{\De}{\mathsf{D}}

\newcommand{\Exc}{\mathsf{Exc}}

\newcommand{\Fin}{\mathsf{Fin}}

\newcommand{\logSH}{\mathsf{logSH}}

\newcommand{\MH}{\mathsf{H}}

\newcommand{\Mod}{\mathsf{Mod}}
\newcommand{\MS}{\mathsf{MS}}
\newcommand{\Perf}{\mathsf{Perf}}
\newcommand{\PH}{\mathbf{P}\mathsf{H}}
\newcommand{\PrL}{\mathsf{Pr}^\mathrm{L}}
\newcommand{\PrLst}{\mathsf{Pr}^{\mathrm{L}}_\mathrm{st}}
\newcommand{\PS}{\mathbf{P}\mathsf{S}}
\newcommand{\PSh}{\mathsf{P}}
\newcommand{\Ring}{\mathsf{Ring}}
\newcommand{\Sch}{\mathsf{Sch}}

\newcommand{\Sh}{\mathsf{Sh}}
\newcommand{\Sm}{\mathsf{Sm}}
\newcommand{\Sp}{\mathsf{Sp}}
\newcommand{\Span}{\mathsf{Span}}
\newcommand{\Stack}{\mathsf{Stack}}

\newcommand{\Vect}{\mathsf{Vect}}

\newcommand{\bs}{\mathrm{b}}
\newcommand{\B}{\mathrm{B}}
\newcommand{\Bl}{\mathrm{Bl}}

\newcommand{\C}{\mathrm{C}}

\newcommand{\dR}{\mathrm{dR}}
\newcommand{\dRhc}{\hat{\mathrm{dR}}}
\newcommand{\cc}{\mathrm{L}}
\newcommand{\cn}{\mathrm{N}}

\newcommand{\crys}{\mathrm{crys}}
\newcommand{\D}{\mathrm{D}}

\newcommand{\emb}{\mathrm{emb}}
\newcommand{\et}{\mathrm{\acute{e}t}}
\newcommand{\E}{\mathrm{E}}

\newcommand{\fin}{\mathrm{Fin}}
\newcommand{\gp}{\mathrm{gp}}
\newcommand{\ho}{\mathrm{H}}
\newcommand{\hypo}{\mathrm{hypo}}
\newcommand{\I}{\mathrm{I}}

\newcommand{\ji}{\mathrm{j}}
\newcommand{\K}{\mathrm{K}}
\newcommand{\Nis}{\mathrm{Nis}}
\newcommand{\op}{\mathrm{op}}

\newcommand{\ps}{\mathrm{P}}
\newcommand{\Q}{\mathrm{Q}}
\newcommand{\qbu}{\mathrm{qbu}}

\newcommand{\RG}{\mathrm{R}\Gamma}
\newcommand{\Sone}{\mathrm{S}^1}

\newcommand{\syn}{\mathrm{syn}}
\newcommand{\thom}{\mathrm{t}}

\newcommand{\un}{\mathrm{un}}
\newcommand{\vt}{\mathrm{v}}
\newcommand{\vect}{\mathrm{Vect}}
\newcommand{\ws}{\mathrm{S}}
\newcommand{\z}{\mathrm{z}}
\newcommand{\Zar}{\mathrm{Zar}}

\newcommand{\fD}{\mathfrak{D}}
\newcommand{\fN}{\mathfrak{N}}

\DeclareMathOperator{\id}{id}
\DeclareMathOperator{\Hom}{Hom}
\DeclareMathOperator{\Map}{Map}

\DeclareMathOperator*{\colim}{colim}
\DeclareMathOperator{\Spec}{Spec}

\DeclareMathOperator{\fib}{fib}
\DeclareMathOperator{\cofib}{cofib}

\DeclareMathOperator{\fil}{Fil}

\DeclareMathOperator{\LSym}{LSym}
\DeclareMathOperator{\Th}{Th}

\DeclareMathOperator{\Proj}{Proj}

\DeclareMathOperator{\gys}{gys}
\DeclareMathOperator{\Witt}{W}
\DeclareMathOperator{\dlog}{dlog}

\DeclareMathSymbol{\jj}{\mathalpha}{letters}{"11}

\usepackage{relsize}
\usepackage[bbgreekl]{mathbbol}
\usepackage{amsfonts}
\DeclareSymbolFontAlphabet{\mathbb}{AMSb} 
\DeclareSymbolFontAlphabet{\mathbbl}{bbold}

\begin{document}

\title{The $\bP^1$-motivic Gysin map}
\author{Longke Tang}

\begin{abstract}
We develop a $\bP^1$-unstable non-$\bA^1$-invariant theory of motivic spaces and spectra, and construct the Gysin map therein for regular immersions. This in particular gives the Gysin map in the Annala--Hoyois--Iwasa $\bP^1$-motivic spectra, and thus gives a uniform construction for the Gysin maps of various cohomology theories. 
\end{abstract}
 
\maketitle

\setcounter{tocdepth}{1}

\tableofcontents

\section{Introduction}

The \emph{$\bP^1$-motivic homotopy theory} is a version of motivic homotopy theory recently developed by Annala--Hoyois--Iwasa \cite{annala-hoyois-iwasa-2023}. It is a generalization of Voevodsky's $\bP^1$-stable $\bA^1$-homotopy theory \cite{voevodsky-motivic-2002}, in that it keeps $\bP^1$-stability but replaces $\bA^1$-invariance with the weaker condition called \emph{smooth blowup excision} \cite[\S 2]{annala-hoyois-iwasa-2023}, cf.\ Definition \ref{definition-qbu}. This makes it suitable for studying many non-$\bA^1$-invariant cohomology theories, e.g.\ Hodge cohomology \cite[\texttt{0FM4}]{stacks}, de Rham cohomology \cite[\texttt{0FL6}]{stacks} away from characteristic $0$, crystalline cohomology \cite[\texttt{07GI}]{stacks}, and the recently developed prismatic cohomology \cite{bhatt-scholze-prismatic}, as they all satisfy the smooth blowup excision \cite[\S 9.4]{bhatt-lurie-absoluteprismaticcohomology}. 

Unlike the study of $\bA^1$-homotopy theory which has been carried out for decades and successfully applied to solve conjectures including Bloch--Kato \cite{voevodsky-2011-Zl-coefficients}, the study of $\bP^1$-motivic homotopy theory has just got started: there are many basic constructions yet to be done and many structures yet to be revealed, among which stands the \emph{Gysin map}, the topic of this paper. 

\subsection{The Gysin map}
In algebraic geometry, the \emph{Gysin map} typically refers to a wrong-way map between cohomologies associated to a closed immersion of varieties. More specifically, let $i\colon Y\to X$ be a closed immersion between schemes smooth over a common base $S$ and let $\RG$ be any Weil cohomology, e.g.\ \'etale, Hodge, or crystalline. Then the Gysin map has the form
$$\gys(i)\colon\RG(Y)(-?)\to\RG(X),$$
where $(-?)$ denotes the relevant twist, typically a cohomological shift of degree twice the codimension of $i$, possibly followed by some Tate twist. It is an important construction for any cohomology theory: for example, it gives the \emph{cycle class} of $Y$ in $\ho^?(X)$ by plugging in $1\in\ho^0(Y)$. However, classically it was constructed separately for various cohomology theories, each by a very nontrivial process. 

Now suppose that we were doing manifold topology instead of algebraic geometry. Then the Gysin map is actually very easy to construct: for a closed embedding $i\colon Y\to X$ of oriented manifolds of codimension $c$, its Gysin map comes from the \emph{Gysin--Thom isomorphism} \cite{gysin1942}
$$\RG(\Th_Y(\N_i))=\RG(Y)[-c],$$
where $\Th_Y(\N_i)$ denotes the \emph{Thom space} of the \emph{normal bundle} $\N_i$, along with the \emph{Pontryagin--Thom collapse map} $X\to\Th_Y(\N_i)$ provided by the \emph{tubular neighborhood lemma}, which says that for certain open neighborhoods $U$ of $Y$ in $X$,
$$\Th_Y(\N_i)=\bar{U}/\partial U=\frac{X}{X\setminus U}.$$
Putting together gives the Gysin map $\RG(Y)[-c]=\RG(\Th_Y(\N_i))\to\RG(X)$. Note that $\Th_Y(\N_i)$ is no longer a manifold but only a pointed space, and we are tacitly using that our $\RG$ (e.g.\ singular cohomology) is also defined for pointed spaces.

Of course, in algebraic geometry, there is no tubular neighborhood lemma. Fortunately, in $\bA^1$-homotopy theory, Morel--Voevodsky \cite[\S 3, Theorem 2.23]{morel-voevodsky-a1} has proved an analogous statement
$$\Th_Y(\N_i)=\frac{X}{X\setminus Y},$$
giving an analog of the Pontryagin--Thom collapse map in their category of $\bA^1$-motivic spaces, and thus a unified construction for the Gysin maps for all $\bA^1$-invariant cohomology theories. 

Unfortunately, for non-$\bA^1$-invariant cohomology theories like Hodge and crystalline, the Gysin maps are still constructed separately in the literature (see for example \cite[\texttt{0G82}]{stacks}), making it difficult to compare the Gysin maps of different cohomology theories and to construct them for new cohomology theories like prismatic cohomology. This paper, or this project, aims to solve the issue once and for all, by constructing a Pontryagin--Thom collapse map in the non-$\bA^1$-invariant motivic homotopy theory recently developed by Annala--Hoyois--Iwasa \cite{annala-hoyois-iwasa-2023}: 

\begin{theorem}[{Construction \ref{construction-gysin}}]\label{gysin-introduction}
    Let $i\colon Y\to X$ be a closed immersion between schemes smooth over a common base $S$. Then there is a canonical map
    $$\gys(i)\colon X\to\Th_Y(\N_i)$$
    in the Nisnevich version of the stable $\infty$-category $\MS_S$ of $\bP^1$-spectra in \cite{annala-hoyois-iwasa-2023}, which is functorial in excessive maps (Definition \ref{definition-excessive-map-arrow}) of closed immersions. In particular, pulling back along the open complement gives a canonical map
    $$\frac{\gys(i)}{\gys(i_{X\setminus Y})}\colon\frac{X}{X\setminus Y}\to\Th_Y(\N_i).$$
\end{theorem}

\begin{remark}
    In contrast with the $\bA^1$-invariant situation, the above map is in general \emph{not} an isomorphism. In fact, it being an isomorphism implies $\bA^1$-invariance: take $i$ to be the closed immersion $\infty\to\bP^1$; then the map becomes
    $$\bP^1/\bA^1\to\Th_Y(\N_{\infty/\bP^1})=\bP^1/0,$$
    which is an isomorphism if and only if the map $\bA^1\to0$ is, which is $\bA^1$-invariance. 
\end{remark}

We also establish basic properties of our Gysin map: 

\begin{theorem}\label{gysin-properties-introduction}
    The Gysin map in Theorem \ref{gysin-introduction} satisfies: 
    \begin{description}
        \item[Symmetric monoidality (Theorem \ref{gysin-symmetric-monoidal})] for a finite family of closed immersions $(i^\omega\colon Y^\omega\to X^\omega)_{\omega\in\Omega}$ of schemes smooth over $S$, there is a canonical commutative diagram
        $$\begin{tikzcd}
            \prod_{\omega\in\Omega}X^\omega\ar[rr,"\bigotimes_{\omega\in\Omega}\gys(i^\omega)"]\ar[d,equal]&&\bigotimes_{\omega\in\Omega}\Th_{Y^\omega}(\N_{i^\omega})\ar[d]\\
            \prod_{\omega\in\Omega}X^\omega\ar[rr,"\gys\left(\prod_{\omega\in\Omega}i^\omega\right)"']&&\Th_{\prod_{\omega\in\Omega}Y^\omega}\left(\N_{\prod_{\omega\in\Omega}i^\omega}\right)
        \end{tikzcd}$$
        functorial with respect to $\Omega$, where the $\Pi$ and $\otimes$ above are both over $S$. 
        \item[Normalization (Theorem \ref{gysin-normalization})] when $X=\bP_Y(\cE\oplus\cO)$ for a vector bundle $\cE/Y$ and $i$ is the immersion $Y=\bP_Y(\cO)\to\bP_Y(\cE\oplus\cO)$, the Gysin map is the quotient map
        $$\bP_Y(\cE\oplus\cO)\to\frac{\bP_Y(\cE\oplus\cO)}{\bP_Y(\cE)}=:\Th_Y(\cE).$$
        \item[Composition (Theorem \ref{theorem-composition-gysin})] for a chain of two closed immersions $Z\xrightarrow{j}Y\xrightarrow{i}X$ of schemes smooth over $S$, there is a canonical commutative square
        $$\begin{tikzcd}
            X\ar[rrr,"\gys(ij)"]\ar[d,"\gys(i)"']&&&\Th_Z(\N_{ij})\ar[d,"\gys(\Th_Z(\N_{Z/i}))"]\\
            \Th_Y(\N_i)\ar[rrr,"\gys(\Th_j(\N_i))"']&&&\Th_Z(j^*\N_i)\otimes_Z\Th_Z(\N_j)
        \end{tikzcd}.$$
    \end{description}
\end{theorem}

Since the construction of our Gysin map relies heavily on the blowup excision of $\bP^1$-spectra, we have to use various identities between nested blowups in order to prove Theorem \ref{gysin-properties-introduction}. We are unable to find these identities in the literature at least in the derived case, so we include the basic theory of blowups in derived algebraic geometry as well as the proof of these identities in Appendix \ref{appendix-dag}. 

\begin{remark}
    Of course, the \textbf{Composition} in Theorem \ref{gysin-properties-introduction} is handicapped from the $\infty$-categorical viewpoint: a full composition compatibility ought to contain larger commutative diagrams for arbitrarily long chains of closed immersions. However, the full composition is already extremely complicated even in the $\bA^1$-invariant case as in \cite{deglisefeldjin2026homotopycoherentgysinfunctoriality}: it needs a significantly heavier theory of nested blowups than what we want to develop here in Appendix \ref{appendix-dag}, so we leave it to our future work. 
\end{remark}

\subsection{A \texorpdfstring{$\bP^1$}{P1}-unstable theory}
Compared to the $\bA^1$-homotopy theory, another missing aspect of the Annala--Hoyois--Iwasa $\bP^1$-motivic homotopy theory is the \emph{unstable} theory, i.e.\ a theory where the pointed $\bP^1$ is not tensor-invertible. During the study of the Gysin map, we notice that we have only used blowup excision and $\bP^1$-homotopy, but not $\bP^1$-stability anywhere, so we propose a $\bP^1$-unstable theory of non-$\bA^1$-invariant motives here: 

\begin{theorem}[Definition \ref{definition-motivic-space}; Propositions \ref{compare-ph-mv}, \ref{compare-ph-ms}]
    For any scheme $S$, there is a pointed presentably symmetric monoidal $\infty$-category $\PH_*(S)$ of \emph{$\bP$-unstable motivic spaces}, which supports canonical presentably symmetric monoidal functors $\PH_*(S)\to\MH_*(S)$ and $\PH_*(S)\to\MS^\un_S$, to the Morel--Voevodsky's category of pointed $\bA^1$-motivic spaces and the Annala--Hoyois--Iwasa's category of unstable $\bP^1$-motivic spectra, respectively. 

    Moreover, let $\PS(S)=\PH_*(S)\otimes\Sp$ be the $\Sone$-stabilization. Then the Gysin map in Theorem \ref{gysin-introduction} and the isomorphisms in \ref{gysin-properties-introduction} all come from $\PS(S)$ through the stabilized version $\PS(S)\to\MS_S$ of the above functor. 
\end{theorem}

\begin{remark}
    In non-$\bA^1$-invariant motivic homotopy theory, the invertibility of $\bP^1$ does not imply that of $\Sone$, resulting in two stabilities as listed below:
    \begin{center}
        \begin{tabular}{c|c|c}
            Category & $\Sone$-stable & $\bP^1$-stable \\
            \hline
            $\PH_*$ & No & No \\
            $\PS$ & Yes & No \\
            $\MS^\un$ & No & Yes \\
            $\MS$ & Yes & Yes
        \end{tabular}
    \end{center}
    
    In $\bA^1$-motivic homotopy theory, $\bP^1$-stability implies $\Sone$-stability, and \cite{morel-voevodsky-a1} managed to construct the Gysin map even before $\Sone$-stabilization. However, in our non-$\bA^1$-invariant theory, we are only able to do so after $\Sone$-stabilization. 
\end{remark}

\subsection{Applications}
Our construction yields Gysin maps for any cohomology theory satisfying mild conditions, and coincides with any existing construction:

\begin{theorem}[Constructions \ref{represent-cohomology-4ff}, \ref{gysin-general-coefficient}]
    Let $\cD\colon\Sch\to\PrLst$ be a presheaf of stable presentably symmetric monoidal categories on derived schemes. Assume:
    \begin{enumerate}
        \item $\cD$ is a Zariski sheaf; 
        \item $\cD$ satisfies quasismooth blowup excision, in the sense that for every blowup
        $$\begin{tikzcd}
            \E_YX\ar[r]\ar[d]\ar[rd,"e" description]&\Bl_YX\ar[d,"b"]\\
            Y\ar[r,"i"']&X
        \end{tikzcd}$$
        where $i$ is a quasismooth closed immersion, and for every $C\in\cD(X)$, we have $C=i_*i^*C\times_{e_*e^*C}b_*b^*C$;
        \item for every $X$ smooth over $S$ and $C\in\cD(X)$, denote $C\langle-1\rangle=\cofib(C\to p_*p^*C)$ where $p\colon\bP^1_X\to X$ is the projection; then $C\langle-1\rangle=C\otimes1_X\langle-1\rangle$, and the formation of $\langle-1\rangle$ commutes with pullbacks. 
    \end{enumerate}
    Then for every quasismooth closed immersion $i\colon Y\to X$ of quasismooth derived $S$-schemes and every $C\in\cD(X)$ there is a natural Gysin map
    $$\gys_C(i)\colon i_*((i^*C)\langle-\N_i\rangle)\to C$$
    where $\N_i$ is the conormal module, and for $B\in\cD(Y)$ write $B\langle-\N_i\rangle=\fib(p_*p^*B\to q_*q^*B)$ where $p\colon\bP_Y(\N_i\oplus\cO)\to Y$ and $q\colon\bP_Y(\N_i)\to Y$ are the projections.
\end{theorem}

In Constructions \ref{represent-cohomology-4ff} and \ref{gysin-general-coefficient} in the main text, the setting will be more general and the conditions will be weaker. We will also provide a general comparison theorem, Theorem \ref{theorem-oriented-comparision}, which in particular implies, for \'etale cohomology, that the above theorem generalizes Gabber's \'etale Gysin map \cite{fujiwara2002absolutepuritygabber}, and for Hodge cohomology, that the above theorem generalizes the Gysin map in \cite[\texttt{0G82}]{stacks} (Example \ref{represent-classical-cohomologies}). 

Besides, during the preparation of this paper, various applications of Theorems \ref{gysin-introduction} and \ref{gysin-properties-introduction} have already been out. See \cite{annalahoyoisiwasa2024atiyahdualitymotivicspectra}, \cite{hoyois2024remarksmotivicsphere}, \cite{annala2025motivicsteenrodproblemaway} and \cite{annala2025motivicsteenrodoperationscharacteristic}. We list a few of their theorems here, which use our Gysin map as an essential input: 

\begin{theorem}[{\cite[Theorem 1.1]{annalahoyoisiwasa2024atiyahdualitymotivicspectra}}]
    For any scheme $S$ and any projective smooth $S$-scheme $X$, $X$ is dualizable in $\MS_S$ with dual $\Th_X(-\Omega_{X/S})$. 
\end{theorem}

Not only does this theorem prove a Poincar\'e duality for virtually any cohomology theory for projective smooth schemes, it also has further arithmetic application \cite{carmeli2025prismaticsteenrodoperationsarithmetic} which really uses that the duality holds already in $\MS$. 

\begin{theorem}[{\cite[Theorem 1.7]{annalahoyoisiwasa2024atiyahdualitymotivicspectra}},\,{\cite[Theorem C]{park2026poincaredualitylogarithmicmotivic}}]
    Let $k$ be a perfect field. Then there is a canonical $\bA^1$-invariant cohomology theory $\RG_{\crys,\bA^1}\colon\Sm_k^\op\to\De(\Witt(k))$ such that whenever $U=X\setminus\partial X$ with $X$ projective smooth and $\partial X$ a strict normal crossing divisor, there is a canonical isomorphism
    $$\RG_{\crys,\bA^1}(U)=\RG_{\crys,\log}(X,\partial X).$$
    In short, log crystalline cohomology is independent of the good compactification. 
\end{theorem}

In \cite{annalahoyoisiwasa2024atiyahdualitymotivicspectra}, this theorem is a clever application of the above dualizability which relies on our Gysin map. In \cite{park2026poincaredualitylogarithmicmotivic}, the proof is independent of this paper, using a version of the Gysin map in the log motivic setting constructed in that paper. 

\subsection{Notation and convention}\label{dag-notation}
We summarize our global notation and convention here, mainly regarding derived algebraic geometry. In short, everything is animated or derived unless otherwise stated, and readers familiar with the setup can safely skip this subsection. Precisely:
\begin{enumerate}
    \item A \emph{category} means an $(\infty,1)$-category. For $n\in\NN$, an \emph{$n$-category} means an $(\infty,n)$-category as in \cite{unicity-barwick} for example. We will at most use $(\infty,2)$-categories in this paper. 
    \item For $n\in\NN$, we use $[n]$ to denote the poset $(\{0,1,\ldots,n\},\le)$, often viewed as a category. 
    \item A \emph{ring} is an animated ring. 
    \item For a ring $R$, an $R$-\emph{module} is an animated $R$-module, or in other words an object in $\De_{\ge0}(R)$. The category $\De_{\ge0}(R)$ is left complete prestable and canonically presentably symmetric monoidal under $\otimes_R$. We denote the presentable stabilization of $\De_{\ge0}(R)$ by $\De(R)$ and call its objects \emph{$R$-complexes}. 
    \item For a ring $R$, an $R$-\emph{algebra} is a ring along with a map from $R$ to it. We let $\Alg(R)$ denote the category of $R$-algebras, which is canonically presentably symmetric monoidal under $\otimes_R$. 
    \item A \emph{surjection} of rings, modules, or algebras means a map that is surjective on classical truncations, i.e.\ on $\pi_0$.
    \item A \emph{prestack} is an accessible functor $\Ring\to\Ani$. For a prestack $X$ and a ring $R$, an \emph{$R$-point} of $X$ is a point in the anima $X(R)$.
    \item For an accessible topology $\tau$ on $\Ring$, a \emph{$\tau$-stack} is a prestack satisfying $\tau$-descent. A \emph{stack} means a Zariski stack in this paper.
    \item An \emph{affine scheme} is a representable prestack. For a ring $R$, we let $\Spec(R)$ denote the prestack represented by $R$, i.e.\ the functor $\Map(R,-)$, and we call it the \emph{spectrum} of $R$.
    \item A \emph{scheme} is a stack that is Zariski covered by Zariski open substacks that are affine schemes.
    \item For a prestack $X$, an \emph{$\cO_X$-module} is a natural transformation $X\to\Mod$. $\cO_X$-modules form a presentably symmetric monoidal left complete prestable category denoted $\De_{\ge0}(X)$ as each $\De_{\ge0}(R)$ is one. Similarly, we denote the presentable stabilization of $\De_{\ge0}(X)$ by $\De(X)$, and call its objects \emph{$\cO_X$-complexes}. Note that this notion is the derived analog of \emph{quasicoherent $\cO_X$-modules} rather than general $\cO_X$-modules for any topos built from $X$ and the structure sheaf $\cO_X$ thereon. We choose to omit the word ``quasicoherent'' because the latter notion does not make sense for a general prestack, and also because we only care about quasicoherent ones. 
    \item From the previous item, for a prestack map $Y\to X$, an $\cO_X$-module $\M$ naturally pulls back to an $\cO_Y$-module, which we denote by $\M|_Y$. 
    \item For a prestack $X$, an \emph{$\cO_X$-algebra} is a natural transformation $X\to\Alg$. $\cO_X$-algebras form a presentably symmetric monoidal category as each $\Alg(R)$ is one. For an $\cO_X$-algebra $\cA$, its \emph{relative spectrum}, denoted $\underline{\Spec}(\cA)$, is the prestack over $X$ whose $R$-point is the datum of an $x\in X(R)$ plus an $R$-algebra map $\cA(x)\to R$. The structure map $\underline{\Spec}(\cA)\to X$ is clearly representable by affine schemes. Conversely, every prestack map representable by affine schemes is a relative spectrum. 
    \item A \emph{closed immersion} of prestacks is a prestack map representable by ring surjections. Equivalently a closed immersion $i\colon Y\to X$ is the relative spectrum of an $\cO_X$-algebra $i_*\cO_Y$ where the structure map $\cO_X\to i_*\cO_Y$ is a surjection. In this case, we call the $\cO_X$-module $\fib(\cO_X\to i_*\cO_Y)$ the \emph{ideal} of the closed immersion $i$, and we often just write $\cO_Y$ to mean $i_*\cO_Y$. We will not try to define the notion of ideals seriously, but will rather view ideals only as modules with a map to $\cO_X$. In classical algebraic geometry, a closed immersion is a monomorphism, so one often calls it a \emph{closed substack}. In contrast, in derived algebraic geometry, a general closed immersion is almost never a monomorphism, so we will avoid using the terminology \emph{closed substack}. 
    \item\label{dag-notation:divisor} A \emph{divisor} is a closed immersion whose ideal is an invertible module. This generalizes the notion of an \emph{effective Cartier divisor} in classical algebraic geometry to derived algebraic geometry. 
    \item\label{dag-notation:adding-closed-immersion} For two closed immersions $Y\to X$ and $Z\to X$, we let $Y+Z\to X$ denote the closed immersion given by the spectrum of the $\cO_X$-algebra $\cO_Y\times_{\cO_{Y\times_XZ}}\cO_Z$, whose structure map is easily seen to be a surjection. More generally, for a finite family $(Y_\omega\to X)_{\omega\in\Omega}$ of closed immersions, denote by $\sum_{\omega\in\Omega}Y_\omega$ the closed immersion given by the $\cO_X$-algebra $\lim_{\varnothing\ne\Psi\subseteq\Omega}\cO_{Y_\Psi}$, where $Y_\Psi$ is the fiber product of $(Y_\omega)_{\omega\in\Psi}$ over $X$ (see Example \ref{example-product-excessive}). This defines a commutative monoid structure on closed immersions to $X$, where divisors form a commutative submonoid. 
    \item For a ring $A$ and an $A$-module $M$, we define $\bV_A(M)=\Spec(\LSym^\bullet(M))$ and call it the \emph{total space} of $M$. By definition, its $R$-point is the data of a ring map $A\to R$ and an $R$-linear map $M\otimes_AR\to R$. This construction globalizes to an $\cO_X$-module $\M$ for a prestack $X$, which we denote by $\bV_X(\M)$ or $\bV(\M)$. 
    \item For a ring $A$ and an $A$-module $M$, we define $\bP_A(M)=\Proj(\LSym^\bullet(M))$ and call it the \emph{projective space} of $M$. By definition, its $R$-point is the data of a map $A\to R$, an invertible $R$-module $L$, and a surjection $M\otimes_AR\to L$. This is a scheme, and is proper over $\Spec(A)$ if $M$ is finitely generated. This construction globalizes to an $\cO_X$-module $\M$ for a prestack $X$, which we denote by $\bP_X(\M)$ or $\bP(\M)$. 
    \item For a ring map $R\to A$, the \emph{cotangent module} $\cc_{A/R}$ is the $A$-module whose map to an $A$-module $M$ amounts to a section of the trivial square-zero extension $A\oplus M\to A$ treated as an $R$-algebra map. When $R\to A$ is surjective, $\pi_0\cc_{A/R}=0$, and we call $\cc_{A/R}[-1]$ the \emph{conormal module} and denote it by $\cn_{A/R}$. These constructions globalize at least to prestack maps representable by schemes, which we denote by $\cL_{Y/X}$ and $\N_{Y/X}$ for such a map $Y\to X$. This usage of the letter N might seem strange, but it aligns with the previous two items so that a $Y$-point in $\bV_Y(\N_{Y/X})$ is a normal vector field of $Y$ in $X$. 
    \item\label{dag-notation:pseudocoherent} For a ring $R$, we say that an $R$-complex $M$ is \emph{pseudocoherent} if for every $n\in\ZZ$, there exists a perfect complex $M_n$ and a map $M_n\to M$ whose cofiber is $n$-connective. We say that an $R$-algebra $A$ is \emph{pseudocoherent} if there exists a $d\in\NN$ and an $R$-algebra map $R[x_1,\ldots,x_d]\to A$ such that $A$ is pseudocoherent as an $R[x_1,\ldots,x_d]$-module. This notion globalizes at least to prestack maps representable by schemes. This notion is called \emph{almost of finite presentation} in \cite[Definition 7.2.4.26]{ha} and \cite[Definition 4.2.0.1]{sag}, but we choose to follow \cite[\texttt{067G},\texttt{067X}]{stacks} and call it \emph{pseudocoherence} to avoid potential confusion with almost mathematics. 
    \item\label{dag-notation:d-smooth} Following \cite[Definition 2.18]{tang2024slicingcriterionindsmoothring}, we say that a ring map $R\to A$ is \emph{$d$-smooth}, if it is pseudocoherent and its cotangent complex $\cc_{A/R}$ has tor-amplitude $[0,d]$. Moreover, \emph{smooth} means $0$-smooth, and \emph{quasismooth} means $1$-smooth. This notion globalizes at least to prestack maps representable by schemes. 
\end{enumerate}

The above is not to introduce derived algebraic geometry but only to resolve possible ambiguity. For introductions, see \cite[\S 5.1]{cesnaviciusscholze2024purityflat}, \cite[Part 1]{jerc2025derivedalgebraicgeometrydmodules}, and \cite{sag}. 

\subsection{Acknowledgements}
I apologize for the long delay and the lack of higher composition coherence in this paper. The construction of the Gysin map itself came to me on a transatlantic flight in Spring 2023, but the complexity of the structures needed for proving its properties went out of my expectations several times, until I finally became content with a two-map composition compatibility, realizing that I would not be able to write down the full composition compatibility timely before my graduation. I thank Marc Hoyois for his continuous encouragement and our extensive discussions regarding this project, which essentially motivate my proofs for both the normalization and the composition of the Gysin map. I thank Toni Annala for many valuable discussions and many comments on a preliminary draft. I thank Elden Elmanto, Lars Hesselholt, Ryomei Iwasa, and Noam Zimhoni for valuable discussions and comments on an earlier draft. Finally, I thank Bhargav Bhatt, Federico Binda, Tess Bouis, Dustin Clausen, Alessandro D'Angelo, Rune Haugseng, Yuanyang Jiang, Niklas Kipp, Ishan Levy, Shizhang Li, Yixiao Li, Jacob Lurie, Shiji Lyu, Linquan Ma, Lucas Mann, Akhil Mathew, Arnaud Mayeux, Arjun Nigam, Arpon Raksit, Brian Shin, Vladimir Sosnilo, Gleb Terentiuk, Vadim Vologodsky, Qixiang Wang, Kirsten Wickelgren, Yuchen Wu, Wei Yang, Bogdan Zavyalov, Daming Zhou, and Xinwen Zhu for helpful discussions. 

This work is my PhD thesis at Princeton University. I want to thank my advisor Bhargav Bhatt for his constant encouragement and support, and Princeton University for its support, during my entire PhD life. This work was first announced, and its earliest part was done, during the trimester program on the arithmetic of the Langlands program in Summer 2023. I thank the Hausdorff Research Institute for Mathematics for its hospitality and support during the trimester. A considerable part of this work was done during my several visits to the University of Regensburg. I thank Marc Hoyois for hosting me, the SFB 1085 Higher Invariants for support, and the University of Regensburg for hospitality and support. Besides, many valuable discussions occurred during my visits at various institutes. I would like to thank all of them, as well as my hosts Ruochuan Liu at Peking University, Daxin Xu at the Chinese Academy of Sciences, Hui Gao at the Southern University of Science and Technology, Elden Elmanto at the University of Toronto, Kirsten Wickelgren at Duke University, Toni Annala at the University of Chicago, Tong Liu at Purdue University, Ryomei Iwasa at the University of Copenhagen, Richard Taylor at Stanford University, Bogdan Zavyalov at the University of Maryland, and Vaughan McDonald at the University of California, Berkeley.

\section{\texorpdfstring{$\bP^1$}{P1}-unstable motivic spectra}

In this section, we introduce a minimal version of motivic spectra where the Gysin map is available. It will be very similar to but a little more general than the version in \cite{annala-iwasa-2022} and \cite{annala-hoyois-iwasa-2023}, and will clarify the structures that we actually use in constructing the Gysin map. Impatient readers who know the theory in \cite{annala-hoyois-iwasa-2023} can skip this section and do everything with their motivic spectra, i.e.\ read all the $\cP$ as $\Sm$ and read all the $\PS(\cP_S)$ as the Nisnevich version of $\MS_S$. 

All our schemes are quasicompact and quasiseparated. 

We choose a regular cardinal $\kappa$ which can possibly be $\aleph_0$, and denote by $\PrL_\kappa$ the ($\kappa$-presentable, see \cite[Theorem 2.9]{gestalten}) category of all $\kappa$-presentable categories, as in \cite[Definition 2.7]{gestalten}. 

\subsection{Basic definitions}

\begin{definition}[motivic setup]\label{definition-motivic-setup}
    A \emph{$\qbu$-motivic setup}, or simply a \emph{motivic setup}, is a wide subcategory $\cP\subseteq\Ring^\kappa$, i.e.\  a set $\cP$ of maps of $\kappa$-compact animated rings closed under composition, such that:
    \begin{enumerate}
        \item\label{definition-motivic-setup:product} $\cP$ is closed under base change;
        \item $\cP$ contains open immersions;
        \item $\cP$ is Zariski local on both the source and the target; accordingly, we can talk about a scheme map being in $\cP$, meaning Zariski locally so; 
        \item\label{definition-motivic-setup:projective} $\cP$ contains projective bundles, i.e.\ for any scheme $Y$ and any vector bundle $\cE/Y$, the projective bundle map $\bP(\cE)\to Y$ lies in $\cP$;
        \item\label{definition-motivic-setup:blowup} $\cP$ is closed under quasismooth blowups, i.e.\ for any scheme $S$ and any quasismooth closed immersion $Y\to X$ in $\cP_S$, the blowup $\Bl_YX$ (see \cite[Definition 4.3]{hekkingkhanrydh2025deformationnormalbundleblowups}, denoted as $\Bl_{Y/X}$ loc.\ cit.) lies in $\cP_S$. 
    \end{enumerate}
    Note that in the situation of (\ref{definition-motivic-setup:blowup}) the exceptional divisor $\E_YX$ also lies in $\cP_S$, due to (\ref{definition-motivic-setup:projective}) and quasismoothness: by the discussion at the end of \cite[\S 4.3]{hekkingkhanrydh2025deformationnormalbundleblowups}, $\E_YX=\bP(\N_{Y/X})$ is the projective bundle of a vector bundle. 

    For $R\in\Ring$ not $\kappa$-compact, we define $\cP_R\subseteq\Ring_R^\kappa$ by extending via $\kappa$-filtered colimits from $\kappa$-compact rings, and define $\cP_S$ for general schemes $S$ by descent. 
\end{definition}

\begin{example}\label{example-motivic-setup}
    Examples of motivic setups include:
    \begin{enumerate}
        \item\label{example-motivic-setup:smooth} The class $\Sm$ of smooth morphisms. In this case, for any scheme $S$, every map between smooth $S$-schemes is quasismooth. 
        \item\label{example-motivic-setup:quasi-smooth} More generally, for any $d\in\NN$, the class $d\text{-}\Sm$ of $d$-smooth morphisms as in \S \ref{dag-notation}(\ref{dag-notation:d-smooth}); when $d=1$ this is the usual quasismoothness. 
        \item The class of pseudo-coherent morphisms as in \S \ref{dag-notation}(\ref{dag-notation:pseudocoherent}), if $\kappa>\aleph_0$. 
        \item The class of perfect morphisms as in \cite[\texttt{0685}]{stacks}, if $\kappa>\aleph_0$. Note that the notion obviously generalizes to derived algebraic geometry. 
    \end{enumerate}
    In this paper, whenever we specialize, we only use (\ref{example-motivic-setup:smooth}) or occasionally (\ref{example-motivic-setup:quasi-smooth}). 
\end{example}

From now on, we fix a motivic setup $\cP$. 

\begin{definition}[quasismooth blowup excision]\label{definition-qbu}
    Let $S$ be a scheme. We say that a presheaf on $\cP_S$ satisfies \emph{quasismooth blowup excision}, if for any quasismooth closed immersion $Y\to X$ in $\cP_S$, its blowup square
    $$\begin{tikzcd}
        \E_YX\ar[r]\ar[d]&\Bl_YX\ar[d]\\
        Y\ar[r]&X
    \end{tikzcd}$$
    is sent to a pullback square. We let $\Sh_{\Zar,\qbu}(\cP_S)$ denote the category of Zariski sheaves of animas on $\cP_S$ that satisfy quasismooth blowup excision. In other words, $\Sh_{\Zar,\qbu}(\cP_S)$ is the accessible localization of $\Sh_\Zar(\cP_S)$ along morphisms of the form $\Bl_YX\sqcup_{\E_YX}Y\to X$. 
    
    Note that $\Sh_{\Zar,\qbu}(\cP_S)$ is compactly generated, with a symmetric monoidal structure preserving compact objects induced by the cartesian product of $\cP_S$, as well as a symmetric monoidal functor $\cP_S\to\Sh_{\Zar,\qbu}(\cP_S)$ obtained from Yoneda. We often abuse notation to let $X\in\Sh_\Zar(\cP_S)$ denote its image in $\Sh_{\Zar,\qbu}(\cP_S)$ or the pointed category $\Sh_{\Zar,\qbu}(\cP_S)_*$ (thus omitting the $(-)_+$ notation), and for a map $Y\to X$ in $\Sh_\Zar(\cP_S)$, to let $X/Y$ denote its cofiber in $\Sh_{\Zar,\qbu}(\cP_S)_*$. 
\end{definition}

\begin{remark}
    The reason that we only impose quasismooth blowup excision is very practical: the non-$\bA^1$-invariant cohomology theories that we want to study in practice only satisfy quasismooth blowup excision; see \cite[\S 9.4]{bhatt-lurie-absoluteprismaticcohomology}. Of course, one can impose stronger excision conditions specified by larger classes of closed immersions, study the corresponding motivic theories, and construct the Gysin map as in Construction \ref{construction-gysin} below. However, even if your cohomology satisfies stronger excision conditions, the Thom spectrum appearing in the Gysin map still looks weird when the conormal module is not finite locally free; it seems that one should not think of it as a twist in any sense. 
\end{remark}

\begin{remark}\label{generalization-stack}
    In fact, our entire theory carries over for the base $S$ in Definition \ref{definition-qbu} being merely a stack (or even an analytic stack, as long as the blowup behaves the same way as it does here) rather than a scheme, with the category $\cP_S$ taken as that of stacks over $S$ with structure maps representable by scheme maps in $\cP$. The reason that we only allow a scheme base is twofold:
    \begin{itemize}
        \item The case of a scheme base already illustrates our ideas very well.
        \item We are not sure that the categories $\Sh_{\Zar,\qbu}(\cP_S)$ and $\PS(\cP_S)$ below still capture a good notion of motives over a stack $S$. 
    \end{itemize}
\end{remark}

Now we define our version of unstable $\bP$-motivic spaces and spectra. 

\begin{definition}[$\bP$-unstable motivic space]\label{definition-motivic-space}
    For a scheme $S$, consider the following three $\Sh_\Zar(\cP_S)$-linear functors from $\Sh_\Zar(\cP_S)_{/\B\Gm}$ to $\Sh_{\Zar,\qbu}(\cP_S)_*$:
    \begin{itemize}
        \item $\bs\colon\Sh_\Zar(\cP_S)_{/\B\Gm}\to\Sh_{\Zar,\qbu}(\cP_S)_*$ defined as $\bs(X,\cL)=X$; 
        \item $\vt\colon\Sh_\Zar(\cP_S)_{/\B\Gm}\to\Sh_{\Zar,\qbu}(\cP_S)_*$ defined as $\vt(X,\cL)=\bV_X(\cL)$; 
        \item $\thom\colon\Sh_\Zar(\cP_S)_{/\B\Gm}\to\Sh_{\Zar,\qbu}(\cP_S)_*$ defined as $\thom(X,\cL)=\frac{\bP_X(\cL\oplus\cO)}{\bP_X(\cL)}$; 
    \end{itemize}
    consider also the following three $\Sh_\Zar(\cP_S)$-linear natural transformations: 
    \begin{itemize}
        \item $0\colon\bs\to\vt$ given by the closed immersion of the zero section;
        \item $\ji\colon\vt\to\thom$ given by the natural open immersion $\bV_X(\cL)\to\bP_X(\cL\oplus\cO)$;
        \item $\z\colon\vt\to\thom$ that takes everything to the zero section $\bP_X(\cO)\to\bP_X(\cL\oplus\cO)$;
    \end{itemize}
    obviously $\ji\circ0=\z\circ0$. 
    
    Define $\PH_*(\cP_S)$ as the initial object in $\CAlg(\PrL_\kappa)_{\Sh_{\Zar,\qbu}(\cP_S)_*/}$ equipped with a homotopy $h\colon\ji=\z$ and a homotopy between its composite with $0\colon\bs\to\vt$ and the obvious $\ji\circ0=\z\circ0$ mentioned above, all as $\Sh_\Zar(\cP_S)$-linear natural transformations between functors from $\Sh_\Zar(\cP_S)_{/\B\Gm}$ to $\PH_*(\cP_S)$. Call $\PH_*(\cP_S)$ the category of \emph{$\bP$-unstable motivic spaces} over $S$. Define $\PS(\cP_S)=\PH_*(\cP_S)\otimes\Sp$ and call it the category of \emph{$\bP$-unstable motivic spectra} over $S$. 

    We use the notation $\PH_*$ just to emphasize that it is pointed, not to say that it has an unpointed variant. When $\cP=\Sm$ is the wide subcategory of smooth maps, we denote $\PH_*(\Sm_S)$ by $\PH_*(S)$ and $\PS(\Sm_S)$ by $\PS(S)$ for brevity. 
\end{definition}

\begin{remark}[intuition of $\PH_*$ and $\PS$]\label{f-sharp-p-homotopy}
    Loosely speaking, $\PH_*(\cP_S)$ is obtained from $\Sh_{\Zar,\qbu}(\cP_S)_*$ by freely adding a homotopy between the maps $\ji,\z\colon\bV_X(\cL)\rightrightarrows\frac{\bP_X(\cL\oplus\cO)}{\bP_X(\cL)}$, functorially in $(X,\cL)\in\Sh_\Zar(\cP_S)_{/\B\Gm}$. In particular, in $\PH_*(\cP_S)$, after composed with $\ji$, any section $s\colon X\to\bV_X(\cL)$ is canonically homotopic to the zero section, and when $s=0$ this homotopy is canonically the identity. 

    In $\PS(\cP_S)$, stability gives the splitting $\bP_X(\cL\oplus\cO)=\frac{\bP_X(\cL\oplus\cO)}{\bP_X(\cL)}\oplus\bP_X(\cL)$ due to the projection map $\bP_X(\cL\oplus\cO)\to X=\bP_X(\cL)$. Let $\tilde{\mathrm{\jj}},\tilde{\z}\colon\bV_X(\cL)\rightrightarrows\bP_X(\cL\oplus\cO)$ denote the open immersion and the projection composed with the zero section immersion, respectively. Then under the above splitting, on the component $\frac{\bP_X(\cL\oplus\cO)}{\bP_X(\cL)}$ they become $\ji$ and $\z$, and on the component $\bP_X(\cL)$ they both become the projection map $\bV_X(\cL)\to X=\bP_X(\cL)$. Therefore, in $\PS(\cP_S)$, there is a canonical homotopy $\tilde{\mathrm{\jj}}=\tilde{\z}$ which similarly becomes the obvious $\tilde{\mathrm{\jj}}\circ0=\tilde{\z}\circ0$ after composed with the zero section $0\colon X\to\bV_X(\cL)$, functorially in $(X,\cL)\in\Sh_\Zar(\cP_S)_{/\B\Gm}$, and in particular, after composed with $\tilde{\mathrm{\jj}}$, any section $s\colon X\to\bV_X(\cL)$ is canonically homotopic to the zero section, and when $s=0$ this homotopy is canonically the identity. 
    
    In our Construction \ref{construction-gysin} of the Gysin map, we will only use this fact in the case $\cL=\cO$ and $s=1$. However, we will need the full generality of this fact through Proposition \ref{proposition-unstable-gysin} when proving Theorem \ref{gysin-normalization}, the normalization of the Gysin map (i.e.\ calculating the Gysin map of the zero section immersion into a vector bundle).
\end{remark}

\begin{remark}[well-definedness]\label{PH-well-defined-compactly-generated}
    The existence of the initial object in Definition \ref{definition-motivic-space} is guaranteed by the adjoint functor theorem as follows: 

    Consider the functor $\CAlg(\PrL_\kappa)_{\Sh_{\Zar,\qbu}(\cP_S)_*/}\to\Ani$ that takes a $\Sh_\Zar(\cP_S)_*$-algebra $\cA$ to the anima of homotopies $h\colon\ji=\z$ such that after composed with $0\colon\bs\to\vt$ it becomes the obvious $\ji\circ0=\z\circ0$, all as $\Sh_\Zar(\cP_S)$-linear natural transformations between functors from $\Sh_\Zar(\cP_S)_{/\B\Gm}$ to $\cA$. This functor is clearly $\kappa$-accessible and preserves all limits, so it has a left adjoint, whose value at the singleton is the desired $\PH_*(\cP_S)\in\CAlg(\PrL_\kappa)_{\Sh_{\Zar,\qbu}(\cP_S)_*/}$. 

    For regular cardinals $\kappa<\lambda$, using the adjunction $\PrL_\kappa\rightleftarrows\PrL_\lambda$ where the left adjoint is the inclusion and is symmetric monoidal, we see that $\PH_*(\cP_S)$ is insensitive to the regular cardinal $\kappa$. In particular, taking $\kappa=\aleph_0$, we see that $\PH_*(\cP_S)$ is compactly generated, its symmetric monoidal structure preserves compact objects, and the defining symmetric monoidal functor $\Sh_\Zar(\cP_S)_*\to\PH_*(\cP_S)$ preserves compact objects. Clearly the same hold for $\PS(\cP_S)$. 
\end{remark}

\begin{remark}[$f^*$-functoriality]\label{f-star-functoriality}
    Let $f\colon T\to S$ be a scheme map. Left Kan extending the base change functor from $\cP_S$ to $\cP_T$ gives a symmetric monoidal functor $f^*\colon\Sh_\Zar(\cP_S)\to\Sh_\Zar(\cP_T)$, hence $f^*\colon\Sh_\Zar(\cP_S)_{/\B\Gm}\to\Sh_\Zar(\cP_T)_{/\B\Gm}$ and $f^*\colon\Sh_{\Zar,\qbu}(\cP_S)_*\to\Sh_{\Zar,\qbu}(\cP_T)_*$, and hence $f^*\colon\PH_*(\cP_S)\to\PH_*(\cP_T)$ and $f^*\colon\PS(\cP_S)\to\PS(\cP_T)$, all symmetric monoidal, by functoriality of Definition \ref{definition-motivic-space}. 
\end{remark}

\begin{remark}[$f_\sharp$-functoriality]\label{f-sharp-functoriality}
    Let $S$ be a scheme and $f\colon T\to S$ be in $\cP_S$. Left Kan extending the $\cP_S$-linear adjunction between the forgetful functor $\cP_T\to\cP_S$ and the base change functor $\cP_S\to\cP_T$, we obtain a $\Sh_\Zar(\cP_S)$-linear adjunction
    $$f_\sharp\colon\Sh_\Zar(\cP_T)\rightleftarrows\Sh_\Zar(\cP_S):\!f^*.$$
    Since $\Sh_{\Zar,\qbu}(\cP_T)_*=\Sh_{\Zar,\qbu}(\cP_S)_*\otimes_{\Sh_\Zar(\cP_S)}\Sh_\Zar(\cP_T)$ and $\Sh_\Zar(\cP_T)_{/\B\Gm}=\Sh_\Zar(\cP_S)_{/\B\Gm}\otimes_{\Sh_\Zar(\cP_S)}\Sh_\Zar(\cP_T)$, the above $\Sh_\Zar(\cP_S)$-linear adjunction base changes to a $\PH_*(\cP_S)$-linear adjunction
    $$f_\sharp\colon\PH_*(\cP_T)\rightleftarrows\PH_*(\cP_S):\!f^*,$$
    and similarly a $\PS(\cP_S)$-linear adjunction
    $$f_\sharp\colon\PS(\cP_T)\rightleftarrows\PS(\cP_S):\!f^*.$$
\end{remark}

\begin{remark}[changing topology]\label{changing-topology}
    Let $\tau$ be a topology on $\Ring^\kappa$ finer than the Zariski topology such that every $\tau$-cover can be refined by one with maps in $\cP$. Then we have an obvious variant of Definition \ref{definition-motivic-space} with $\Sh_{\tau,\qbu}(\cP_S)$ in place of $\Sh_{\Zar,\qbu}(\cP_S)$, which in fact results in the category $\PH_*(\cP_S)\otimes_{\Sh_\Zar(\cP_S)}\Sh_\tau(\cP_S)$. For example, $\tau=\Nis$ might be more morally correct for motivic homotopy theory. However, the main results in this paper only use the Zariski descent. 
\end{remark}

\begin{remark}[changing setup]
    Let $\cP'$ be a motivic setup containing $\cP$. Left Kan extending the inclusion $\cP_S\to\cP'_S$ yields symmetric monoidal functors $\PH_*(\cP_S)\to\PH_*(\cP'_S)$ and $\PS(\cP_S)\to\PS(\cP'_S)$ by the same procedure as Remark \ref{f-star-functoriality}. 
\end{remark}

\begin{remark}[an alternative theory]\label{alternative-motivic-space}
    In Definition \ref{definition-motivic-space}, instead of $\Sh_\Zar(\cP_S)_{/\B\Gm}$, it might actually be more natural to consider the symmetric monoidal category $\Sh_\Zar(\cP_S)_{/\Vect^\emb}\in\CAlg(\PrL_\kappa)_{\Sh_\Zar(\cP_S)/}$ defined as follows:
    \begin{itemize}
        \item an object is a pair $(X,\cE)$ with $X\in\Sh_\Zar(\cP_S)$ and $\cE$ a finite locally free $\cO_X$-module;
        \item a map $(Y,\cF)\to(X,\cE)$ is a map $f\colon Y\to X$ in $\Sh_\Zar(\cP_S)$ along with a surjection $f^*\cE\to\cF$ of quasicoherent $\cO_Y$-modules, therefore inducing a closed immersion $\bV_Y(\cF)\to\bV_X(\cE)\times_XY$ of vector bundles over $Y$; 
        \item the symmetric monoidal structure is given by the external direct sum;
        \item the structure map $\Sh_\Zar(\cP_S)\to\Sh_\Zar(\cP_S)_{/\Vect^\emb}$ is given by $X\mapsto(X,0)$;
    \end{itemize}
    then consider the following two maps in $\CAlg(\PrL_\kappa)_{\Sh_\Zar(\cP_S)/}$:
    \begin{itemize}
        \item $\vt\colon\Sh_\Zar(\cP_S)_{/\Vect^\emb}\to\Sh_{\Zar,\qbu}(\cP_S)_*$ defined as $\vt(X,\cE)=\bV_X(\cE)$;
        \item $\thom\colon\Sh_\Zar(\cP_S)_{/\Vect^\emb}\to\Sh_{\Zar,\qbu}(\cP_S)_*$ defined as $\thom(X,\cE)=\frac{\bP_X(\cE\oplus\cO)}{\bP_X(\cE)}$, which is symmetric monoidal by \cite[\S 3]{annala-hoyois-iwasa-2023};
    \end{itemize}
    then freely add a homotopy between the two $\Sh_\Zar(\cP_S)$-linear symmetric monoidal natural transformations $\vt\rightrightarrows\thom$ given by the natural open immersion and the zero section, and define $\PH'_*(\cP_S)\in\CAlg(\PrL_\kappa)_{\Sh_{\Zar,\qbu}(\cP_S)_*/}$ as the resulting category. 

    By construction, there is a map $\PH_*(\cP_S)\to\PH'_*(\cP_S)$ in $\CAlg(\PrL_\kappa)_{\Sh_{\Zar,\qbu}(\cP_S)_*/}$, roughly because Definition \ref{definition-motivic-space} is just using the rank $\le1$ part of $\Sh_\Zar(\cP_S)_{/\Vect^\emb}$ without the symmetric monoidal structure. 

    One reason that we stick to the current Definition \ref{definition-motivic-space} is that, there is no symmetric monoidality statement in \cite[Theorem 4.1]{annala-hoyois-iwasa-2023} (although expected), which is crucial for comparing $\PH'_*$ with $\MS^\un$ as in Proposition \ref{compare-ph-ms} below. In fact, other remarks and propositions in this subsection work well or even better with $\PH'_*$. 
\end{remark}

\subsection{Comparing with other motivic theories}
The following propositions show that our $\bP$-unstable motivic spectrum is more general than both the pointed $\bA^1$-motivic space in \cite{morel-voevodsky-a1} and the $\bP^1$-motivic spectrum in \cite{annala-hoyois-iwasa-2023}. 

\begin{proposition}\label{compare-ph-mv}
    For a scheme $S$, let $\MH(S)$ denote the category of $\bA^1$-motivic spaces, i.e.\ the category of $\bA^1$-invariant Nisnevich sheaves on $\Sm_S$. Let $\MH_*(S)$ denote the pointed category. Then there is a unique $\Sh_\Zar(\Sm_S)_*$-linear symmetric monoidal functor from $\PH_*(S)$ to $\MH_*(S)$, which is moreover functorial in $S$. 
\end{proposition}

\begin{proof}
    It suffices to check the following: 
    \begin{itemize}
        \item Sheaves in $\MH(S)$ satisfy quasismooth blowup excision. By \cite[Proposition 2.2]{annala-hoyois-iwasa-2023}, it suffices to check that the blowup square of $0\colon Y\to\bA^n_Y$ is sent to a pullback square. This is easy, as in this case both $Y\to\bA^n_Y$ and $\bP^{n-1}_Y=\E_Y\bA^n_Y\to\Bl_Y\bA^n_Y=\bV_{\bP^{n-1}_Y}(\cO(1))$ are sent to isomorphisms. 
        \item There is a unique homotopy in $\MH_*(S)$ between the two maps $\ji,\z\colon\bV_X(\cL)\rightrightarrows\frac{\bP_X(\cL\oplus\cO)}{\bP_X(\cL)}$ that restricts to the identity homotopy on the zero section, functorially in $(X,\cL)\in\Sh_\Zar(\Sm_S)_{/\B\Gm}$. This is obvious, since the zero section $0\colon X\to\bV_X(\cL)$ is an equivalence in $\MH(S)$. \qedhere
    \end{itemize}
\end{proof}

\begin{proposition}\label{compare-ph-ms}
    For a scheme $S$, let $\MS^\un(\cP_S)$ denote the presentably symmetric monoidal category obtained from the pointed category $\Sh_{\Zar,\qbu}(\cP_S)_*$ by smash-inverting $\bP^1_S/\infty$, and let $\MS(\cP_S)=\MS^\un(\cP_S)\otimes\Sp$. Then there is a canonical $\Sh_{\Zar,\qbu}(\cP_S)$-linear symmetric monoidal functor from $\PH_*(\cP_S)$ to $\MS^\un(\cP_S)$, and hence one from $\PS(\cP_S)$ to $\MS(\cP_S)$, functorially in $S$. 

    In particular, with the Nisnevich version of the theory in \cite{annala-hoyois-iwasa-2023}, we have a canonical functor from $\PH_*(S)$ to $\MS^\un_S$ and hence one from $\PS(S)$ to $\MS_S$. 
\end{proposition}

\begin{proof}
    This is an immediate consequence of \cite[Theorem 4.1]{annala-hoyois-iwasa-2023}. 
\end{proof}

\begin{remark}[comparing with {\cite[Definition 3.6]{bouis2025beilinsonlichtenbaumphenomenonmotiviccohomology}}]
    Bouis and Kundu define in \cite[Definition 3.6]{bouis2025beilinsonlichtenbaumphenomenonmotiviccohomology} a notion called \emph{deflatability} for presheaves on schemes: for a scheme $S$, they call a presheaf $F$ on $\cP_S$ \emph{deflatable}, if functorially in $T\in\cP_S$ there exists a homotopy between the maps
    $$F(\bP^1_T)\rightrightarrows F(\bA^1_T),$$
    where the two maps are induced by $\bA^1_T=\bP^1_T\setminus\infty\subset\bP^1_T$ and $\bA^1_T\to T\xrightarrow{\infty}\bP^1_T$, respectively. At least a priori, this is a weaker condition (or structure) than our $\bP$-unstable motivic spectra. More precisely, if a presheaf $F$ factors through the canonical functor $\cP_S\to\PS(\cP_S)^\omega$, then $F$ is deflatable, because: 
    \begin{itemize}
        \item for such $F$, Remark \ref{f-sharp-p-homotopy} implies that the map $F(\bP^1_T)\to F(\bA^1_T)$ induced by $\bA^1_T=\bP^1_T\setminus\infty\subset\bP^1_T$ is canonically homotopic to that induced by $\bA^1_T\to T\xrightarrow{0}\bP^1_T$, which is then canonically homotopic to that induced by $\bA^1_T\to T\xrightarrow{\infty}\bP^1_T$ again by Remark \ref{f-sharp-p-homotopy}, because both $0$ and $\infty$ in $\bP^1_T$ factor through the inclusion $\bA^1_T=\bP^1_T\setminus1\subset\bP^1_T$. 
    \end{itemize}
    The author thanks Toni Annala for pointing out \cite[Definition 3.6]{bouis2025beilinsonlichtenbaumphenomenonmotiviccohomology} to him. 
\end{remark}

\begin{remark}
    During the preparation of this paper, Noam Zimhoni has proposed to the author another candidate for an unstable $\bP^1$-motivic theory: 

    Let $\Ani^\hypo\subset\Ani$ denote the full subcategory of animas whose every component has a hypoabelian fundamental group, i.e.\ having no perfect subgroup. Then the inclusion has a left adjoint $(-)^+$ which is Quillen's plus construction; see \cite{nikolaus-plus-construction}. Now consider the initial presentable category $\bP^+(\cP_S)$ fitting into the diagram
    $$\begin{tikzcd}
        \vect_S\ar[r]\ar[d,"\bP(-)"']&\vect_S^+\ar[d]\\
        \Sh_{\Zar,\qbu}(\cP_S)\ar[r]&\bP^+(\cP_S)
    \end{tikzcd}$$
    where $\vect_S$ denotes the anima of vector bundles over $S$ and $\vect_S^+$ its plus construction. This seems to be a plausible proposal because
    $$\begin{bmatrix}
        1&1\\
        0&1
    \end{bmatrix}\in\pi_1(\vect_S,\cO^{\oplus2})$$
    is often killed in $\pi_1(\vect_S^+,\cO^{\oplus2})$, thus giving a $\bP^1$-homotopy between $0$ and $1$ in $\bP^+(\cP_S)$. If this is the case, our $\PH_*(\cP_S)$ should map to $\bP^+_*(\cP_S)$. 
\end{remark}

\subsection{Constructions in \texorpdfstring{$\PH_*$}{PH*}}
In this subsection, we develop constructions in $\bP$-unstable motivic spaces that will be used in our study of Gysin maps, starting with the Thom space. 

\begin{definition}[Thom space]\label{definition-thom-space}
    Let $S$ be a scheme, $X\in\Sh_\Zar(\cP_S)$, and $\cE$ be a finite locally free $\cO_X$-module. We define the \emph{Thom space} $\Th_X(\cE)$ as the quotient $\frac{\bP_X(\cE\oplus\cO)}{\bP_X(\cE)}$ taken in $\Sh_{\Zar,\qbu}(\cP_S)_*$, $\PH_*(\cP_S)$ or $\PS(\cP_S)$. When $X=S$ we often omit the subscript $X$. 

    As mentioned in Remark \ref{alternative-motivic-space}, by \cite[\S 3]{annala-hoyois-iwasa-2023}, this gives a symmetric monoidal functor $\Sh_\Zar(\cP_S)_{/\Vect^\emb}\ni(X,\cE)\mapsto\Th_X(\cE)\in\Sh_{\Zar,\qbu}(\cP_S)_*$ where the category $\Sh_\Zar(\cP_S)_{/\Vect^\emb}$ is defined in that remark, with only immersions of vector bundles as morphisms and the external direct sum as the symmetric monoidal structure. 
\end{definition}

The following is a version of \cite[Proposition 4.9]{annala-hoyois-iwasa-2023} and can be thought of as the Gysin map of the zero section immersion $X\xrightarrow{0}\bV_X(\cE)$. It will play a crucial role in the normalization of the Gysin map (Theorem \ref{gysin-normalization}). 

\begin{proposition}\label{proposition-unstable-gysin}
    Let $S$ be a scheme, $X\in\Sh_\Zar(\cP_S)$, and $\cE$ be a finite locally free $\cO_X$-module. Then the map
    $$\Th_X(\cE)=\frac{\bP_X(\cE\oplus\cO)}{\bP_X(\cE)}\to\frac{\bP_X(\cE\oplus\cO)}{\bP_X(\cE\oplus\cO)\setminus\bP_X(\cO)}=\frac{\bV_X(\cE)}{\bV_X(\cE)\setminus0}$$
    has a canonical retraction in $\PS(\cP_S)$ functorially in $(S,X,\cE)$, in the sense that:
    \begin{itemize}
        \item For surjections $\cE\twoheadrightarrow\cE'$, the retraction is compatible with $\Th_X(\cE')\to\frac{\bV_X(\cE')}{\bV_X(\cE')\setminus0}$ mapping to $\Th_X(\cE)\to\frac{\bV_X(\cE)}{\bV_X(\cE)\setminus0}$ in $\PS(\cP_S)^{[1]}$.
        \item For maps $X'\to X$, the retraction is compatible with $\Th_{X'}(\cE)\to\frac{\bV_{X'}(\cE)}{\bV_{X'}(\cE)\setminus0}$ mapping to $\Th_X(\cE)\to\frac{\bV_X(\cE)}{\bV_X(\cE)\setminus0}$ in $\PS(\cP_S)^{[1]}$.
        \item For maps $g\colon S\to T$ in $\cP$, the retraction is compatible with the identification between $g_\sharp\left(\Th_X(\cE)\to\frac{\bV_X(\cE)}{\bV_X(\cE)\setminus0}\right)$ and $\left(\Th_X(\cE)\to\frac{\bV_X(\cE)}{\bV_X(\cE)\setminus0}\right)\in\PS(\cP_T)^{[1]}$.
        \item For maps $f\colon S'\to S$, the retraction is compatible with the identification between $f^*\left(\Th_X(\cE)\to\frac{\bV_X(\cE)}{\bV_X(\cE)\setminus0}\right)$ and $\left(\Th_{X'}(\cE)\to\frac{\bV_{X'}(\cE)}{\bV_{X'}(\cE)\setminus0}\right)\in\PS(\cP_{S'})^{[1]}$, where $X'=X\times_SS'$.
    \end{itemize}
\end{proposition}

\begin{proof}
    It suffices to factor the open immersion $\bP_X(\cE\oplus\cO)\setminus\bP_X(\cO)\to\bP_X(\cE\oplus\cO)$ through the closed immersion $\bP_X(\cE)\to\bP_X(\cE\oplus\cO)$ canonically in $\PS(\cP_S)$. Since they both factor through $\Bl_{\bP_X(\cO)}\bP_X(\cE\oplus\cO)$, we may as well try to factor the open immersion $\bP_X(\cE\oplus\cO)\setminus\bP_X(\cO)\to\Bl_{\bP_X(\cO)}\bP_X(\cE\oplus\cO)$ through the closed immersion $\bP_X(\cE)\to\Bl_{\bP_X(\cO)}\bP_X(\cE\oplus\cO)$. Note that
    $$\bP_X(\cE\oplus\cO)\setminus\bP_X(\cO)=\bV_{\bP_X(\cE)}(\cO(-1)),\,\Bl_{\bP_X(\cO)}\bP_X(\cE\oplus\cO)=\bP_{\bP_X(\cE)}(\cO(-1)\oplus\cO)$$
    compatible with the open immersion above, and the closed immersion above is the zero section. Now Remark \ref{f-sharp-p-homotopy} gives the desired factorization. 
\end{proof}

\begin{remark}\label{thom-space-additive-alternative}
    Since $\cE\mapsto\frac{\bV(\cE)}{\bV(\cE)\setminus0}$ is obviously symmetric monoidal as a functor $(\Vect_S^\emb,\oplus)\to\Sh_\Zar(\cP_S)_*$, Proposition \ref{proposition-unstable-gysin} gives a symmetric monoidal structure of $\Th(-)\colon(\Vect_S^\emb,\oplus)\to\PS(\cP_S)$. Since the blowup loci occurring in \cite[\S 3]{annala-hoyois-iwasa-2023} are away from the zero section, one can see that this symmetric monoidal structure coincides with the one from Definition \ref{definition-thom-space} after composed with the symmetric monoidal functor $\PH_*(\cP_S)\to\PS(\cP_S)$, as follows:

    For a finite set of objects $\cE=(\cE_i)_{i\in I}$ in $\Vect_S^\emb$, we have a commutative diagram
    $$\begin{tikzcd}
        \prod_{i\in I}\Th(\cE_i)\ar[d]&\bB(\cE)/\partial\bB(\cE)\ar[l]\ar[r]\ar[d]&\Th\left(\bigoplus_{i\in I}\cE_i\right)\ar[d]\\
        \prod_{i\in I}\frac{\bP(\cE\oplus\cO)}{\bP(\cE\oplus\cO)\setminus0}&\frac{\bB(\cE)}{\bB(\cE)\setminus0}\ar[l]\ar[r]&\frac{\bP\left(\bigoplus_{i\in I}\cE_i\oplus\cO\right)}{\bP\left(\bigoplus_{i\in I}\cE_i\oplus\cO\right)\setminus0}
    \end{tikzcd}$$
    in $\Sh_{\Zar,\qbu}(\cP_S)_*$ by \cite[Proposition 3.6]{annala-hoyois-iwasa-2023}, whose horizontal maps are all isomorphisms. The upper row is the symmetric monoidal structure of $\Th(-)$ while the lower one is that of $\frac{\bV(-)}{\bV(-)\setminus0}$. Now Proposition \ref{proposition-unstable-gysin} gives retractions of both the left and the right vertical arrows in $\PS(\cP_S)$. By the definition of retractions, this already identifies the two symmetric monoidal structures of $\Th(-)$, without knowing that the retraction on the left coincides with the one on the right.
\end{remark}

We can also reproduce \cite[Theorem 4.1]{annala-hoyois-iwasa-2023} in our unstable setting. 

\begin{proposition}
    Let $S$ be a scheme, $X\in\Sh_\Zar(\cP_S)$, and $\cE$ be a finite locally free $\cO_X$-module. Then there is a canonical homotopy between $\tilde{\mathrm{\jj}}\colon\bV_X(\cE)\to\bP_X(\cE\oplus\cO)$ and $\tilde{\z}\colon\bV_X(\cE)\to X=\bP_X(\cO)\to\bP_X(\cE\oplus\cO)$ in $\PS(\cP_S)$ functorially in $(S,X,\cE)$. Let $\sigma\colon X\to\bV_X(\cE)$ be a section and let $0\colon X\to\bV_X(\cE)$ denote the zero section. Then there is a canonical homotopy between $\tilde{\mathrm{\jj}}\circ\sigma$ and $\tilde{\mathrm{\jj}}\circ0$ functorially in $(S,X,\cE,\sigma)$, which is the identity when $\sigma=0$. 
\end{proposition}

\begin{proof}
    It suffices to prove the first claim. Consider the map of blowup squares
    $$\begin{tikzcd}
        \bP_X(\cE)\ar[r]\ar[d]&\Bl_0\bV_X(\cE)\ar[d]\\
        X\ar[r,"0"']&\bV_X(\cE)
    \end{tikzcd}\longrightarrow\begin{tikzcd}
        \bP_X(\cE)\ar[r]\ar[d]&\Bl_0\bP_X(\cE\oplus\cO)\ar[d]\\
        X\ar[r,"0"']&\bP_X(\cE\oplus\cO)
    \end{tikzcd}$$
    and note that
    $$\Bl_0\bV_X(\cE)=\bV_{\bP_X(\cE)}(\cO(1)),\,\Bl_0\bP_X(\cE\oplus\cO)=\bP_{\bP_X(\cE)}(\cO(1)\oplus\cO)$$
    compatible with the open immersion above. Therefore, Remark \ref{f-sharp-p-homotopy} gives a canonical factorization of $\Bl_0\bV_X(\cE)\to\Bl_0\bP_X(\cE\oplus\cO)$ through $\bP_X(\cE)$ in $\PS(\cP_S)_{\bP_X(\cE)/}$, which implies the claim by pushing out along $\bP_X(\cE)\to X$. 
\end{proof}

\begin{remark}
    The exact same proof shows that there is a canonical homotopy between $\ji\colon\bV_X(\cE)\to\Th_X(\cE)$ and $\z\colon\bV_X(\cE)\to X=\bP_X(\cO)\to\Th_X(\cE)$ in $\PH_*(\cP_S)$ functorially in $(S,X,\cE)$. However, this does not make Remark \ref{alternative-motivic-space} redundant, because this homotopy is not known to be symmetric monoidal in $(X,\cE)$. 
\end{remark}

\section{Construction and properties}

In this section, we construct the Gysin map in our $\bP$-unstable motivic spectra and establish its basic properties. Throughout this section, we fix a motivic setup $\cP$ and a base scheme $S$. Again, impatient readers can read all the $\PS(\cP_S)$ here as the Nisnevich version of $\MS_S$ in \cite{annala-hoyois-iwasa-2023}. 

\subsection{Construction of the Gysin map}
\begin{construction}[Gysin map]\label{construction-gysin}
    Let $i\colon Y\to X$ be a quasismooth closed immersion in $\cP_S$. Consider the map of blowup squares
    $$\begin{tikzcd}
        \bP_Y(\N_i)\ar[r]\ar[d]&\Bl_YX\ar[d]\\
        Y\ar[r]&X
    \end{tikzcd}\longrightarrow\begin{tikzcd}
        \bP_Y(\N_i\oplus\cO)\ar[r]\ar[d]&\Bl_{Y\times0}(X\times\bP^1)\ar[d]\\
        Y\times0\ar[r]&X\times\bP^1
    \end{tikzcd}$$
    and take its cofiber in $\Sh_{\Zar,\qbu}(\cP_S)_*$. We get a cofiber sequence
    $$\Th_Y(\N_i)\to\Q(i)\to X\wedge\bP^1/0$$
    in $\Sh_{\Zar,\qbu}(\cP_S)_*$, where $\Q(i)=\frac{\Bl_{Y\times0}(X\times\bP^1)}{\Bl_{Y\times0}(X\times0)}$ and we are using $\frac{X\times\bP^1}{X\times0}=X\wedge\bP^1/0$. We call it the \emph{fundamental cofiber sequence} of $i\colon Y\to X$. The closed immersion $X=X\times1\to\Bl_{Y\times0}(X\times\bP^1)$ gives a map $X\to\Q(i)$, and Remark \ref{f-sharp-p-homotopy} gives a nullhomotopy of its composition to $X\wedge\bP^1/0$ in $\PS(\cP_S)$, so we get a map
    $$\gys(i)\colon X\to\Th_Y(\N_i)$$
    in $\PS(\cP_S)$, which we call the \emph{Gysin map}, or the \emph{Pontryagin--Thom collapse map}. 
\end{construction}

\begin{remark}
    Readers familiar with deformation to the normal cone will realize that $\Bl_{Y\times0}(X\times\bA^1)\setminus\Bl_{Y\times0}(X\times0)$ is the usual deformation to the normal cone of $i\colon Y\to X$. Here we are using a compactified version of it. 
\end{remark}

Before looking at the functoriality of Gysin maps with respect to maps of closed immersions, we first recall the notion of excessive maps:

\begin{definition}[excessive map, cf.\ {\cite[Definition 3.30]{hekkingkhanrydh2025deformationnormalbundleblowups}}]\label{definition-excessive-map-arrow}
    We say that a map $f$ of closed immersions
    $$\begin{tikzcd}
        Y'\ar[r]\ar[d]&X'\ar[d,"f"]\\
        Y\ar[r]&X
    \end{tikzcd}$$
    is \emph{excessive}, if the map $f^*\cI_{Y/X}\to\cI_{Y'/X'}$ is surjective. 
\end{definition}

\begin{remark}[comparing with {\cite[Definition 3.30]{hekkingkhanrydh2025deformationnormalbundleblowups}}]\label{remark-compare-conormal-surjective}
    Keep the notation of Definition \ref{definition-excessive-map-arrow}. \cite[Definition 3.30]{hekkingkhanrydh2025deformationnormalbundleblowups} defines excessivity as:
    \begin{itemize}
        \item after classical truncation, the above square becomes a pullback;
        \item the map of conormal sheaves $f'^*\N_{Y/X}\to\N_{Y'/X'}$ is surjective. 
    \end{itemize}
    In this remark, we prove that our Definition \ref{definition-excessive-map-arrow} coincides with theirs. For this, pulling back $Y$ along $f$, one reduces to the case where $X'=X$. One also reduces to the affine case immediately. Therefore, it suffices to prove that for surjections of rings $A\to B\to B'$, the following are equivalent: 
    \begin{enumerate}
        \item\label{remark-compare-conormal-surjective:ideal} $\I_{B/A}\twoheadrightarrow\I_{B'/A}$.
        \item\label{remark-compare-conormal-surjective:conormal} $\pi_0(B)=\pi_0(B')$ and $\cn_{B/A}\otimes_BB'\twoheadrightarrow\cn_{B'/A}$. 
    \end{enumerate}
    Now, using the fiber sequences
    $$\I_{B/A}\to\I_{B'/A}\to\I_{B'/B}$$
    and
    $$\cn_{B/A}\otimes_BB'\to\cn_{B'/A}\to\cn_{B'/B}$$
    one reduces to the case where $A=B$. In this case we want to prove that the following are equivalent: 
    \begin{enumerate}
        \item $\I_{B'/B}$ is $1$-connective.
        \item $\pi_0(B)=\pi_0(B')$ and $\cn_{B'/B}$ is $1$-connective. 
    \end{enumerate}
    Obviously, $1$-connectivity of $\I_{B'/B}$ implies $\pi_0(B)=\pi_0(B')$, so we only need to prove that if $B=\pi_0(B)=\pi_0(B')$, then $1$-connectivities of $\I_{B'/B}$ and $\cn_{B'/B}$ are equivalent. This follows from Lemma \ref{lemma-connectivity-comparison} below. 
\end{remark}

\begin{lemma}\label{lemma-connectivity-comparison}
    Let $A\to B$ be a ring surjection where $\ker(\pi_0(A)\to\pi_0(B))$ is a nilpotent ideal in $\pi_0(A)$. Then for all $n\in\NN$, the following are equivalent: 
    \begin{enumerate}
        \item $\I_{B/A}$ is $n$-connective.
        \item $\cn_{B/A}$ is $n$-connective. 
    \end{enumerate}
\end{lemma}

\begin{proof}
    Since $\ker(\pi_0(A)\to\pi_0(B))$ is nilpotent, one can base change the map along itself, and assume without loss of generality that it has a retraction $B\to A$. Then $\I_{A/B}=\I_{B/A}[1]$ and $\cn_{A/B}=\cn_{B/A}[1]\otimes_BA$, so the lemma is equivalent to the same claim for $B\to A$ with $n$ replaced by $n+1$. This way one can reduce to the case where $\I_{B/A}$ is sufficiently connective, in particular $1$-connective. Suppose $\I_{B/A}$ is $m$-connective but not $(m+1)$-connective for some $m\in\ZZ_+$; we want to prove the same for $\cn_{B/A}$. Now \cite[Proposition 25.3.6.1]{sag} says that the universal derivation
    $$\cofib(A\to B)\otimes_AB\to\cc_{B/A}$$
    has an $(m+3)$-connective fiber, which means that its $(-1)$-shift
    $$\I_{B/A}\otimes_AB\to\cn_{B/A}$$
    has an $(m+2)$-connective fiber, which clearly implies our claim. 
\end{proof}

In fact, the Gysin map is functorial with respect to excessive maps:

\begin{remark}[functoriality]\label{trivial-functoriality-gysin}
    Keep the notation in Construction \ref{construction-gysin}, and let
    $$\begin{tikzcd}
        Y'\ar[r]\ar[d]&X'\ar[d]\\
        Y\ar[r]&X
    \end{tikzcd}$$
    be an excessive map of closed immersions in $\cP_S$. Then by functoriality of the blowup as in Definition \ref{definition-blowup-poset} or \cite[\S 4.5]{hekkingkhanrydh2025deformationnormalbundleblowups}, we see that Construction \ref{construction-gysin} is also functorial in the sense that there is a canonical commutative square
    $$\begin{tikzcd}
        X'\ar[r,"\gys(i')"]\ar[d]&\Th_{Y'}(\N_{i'})\ar[d]\\
        X\ar[r,"\gys(i)"']&\Th_Y(\N_i)
    \end{tikzcd}$$

    In particular, since $\cP$ contains open immersions, we can take $X'=X\setminus Y$ and $Y'=\varnothing$, and the above commutative square becomes
    $$\begin{tikzcd}
        X\setminus Y\ar[r]\ar[d]&0\ar[d]\\
        X\ar[r,"\gys(i)"']&\Th_Y(\N_i)
    \end{tikzcd}$$
    so the Gysin map automatically factors through $\frac{X}{X\setminus Y}$. By abuse of notation, we often use $\gys(i)$ to also denote the resulting map $\frac{X}{X\setminus Y}\to\Th_Y(\N_i)$, and also call it the \emph{Gysin map}. 
\end{remark}

\begin{remark}\label{remark-properties-gysin}
    The rest of this section is devoted to some of the nontrivial properties of the Gysin map, including: 
    \begin{enumerate}
        \item\label{remark-properties-gysin:symmetric-monoidal} symmetric monoidality: for finitely many closed immersions, the Gysin map of their product is the tensor product of their Gysin maps;
        \item\label{remark-properties-gysin:normalization} normalization: the Gysin map of the zero section immersion of a vector bundle is the canonical quotient map; 
        \item\label{remark-properties-gysin:composition} composition: the Gysin map of a composition of closed immersions is the composition of the Gysin maps;
    \end{enumerate}
    (\ref{remark-properties-gysin:symmetric-monoidal}) will be addressed in Theorem \ref{gysin-symmetric-monoidal}, (\ref{remark-properties-gysin:normalization}) will be addressed in Theorem \ref{gysin-normalization}, and the composition-of-two case of (\ref{remark-properties-gysin:composition}) will be addressed in Theorem \ref{theorem-composition-gysin}. The composition-of-$n$ case of (\ref{remark-properties-gysin:composition}) (full coherence of composition) will be addressed in future work. 

    Our proofs of these properties extensively use blowups of more complicated diagrams of closed immersions, whose theory is developed in Appendix \ref{appendix-dag} for our need. We introduce the blowup of a poset of closed immersions in Definition \ref{definition-blowup-poset} for the proof of (\ref{remark-properties-gysin:symmetric-monoidal}), and the pushout-blowup of an excessive square of closed immersions in Definition \ref{definition-pushout-blowup} for the proof of the composition-of-two case of (\ref{remark-properties-gysin:composition}). 
\end{remark}

\subsection{Symmetric monoidality}
For symmetric monoidality, we introduce a version of Construction \ref{construction-gysin} for a finite family of quasismooth closed immersions, which we will call the \emph{multiple deformation space}, as it is used to \emph{multiply} Gysin maps. Beware that this is \emph{not} the generalization of what people called the double deformation space before, which we will instead call the \emph{composite deformation space} in Construction \ref{composite-deformation-space} below, as it is used to \emph{compose} Gysin maps. 

For any finite set $\Omega$, any $\Omega$-family $X=(X^\omega)_{\omega\in\Omega}$ in $\cP_S$, and any $\Psi\subseteq\Omega$, we denote $X^\Psi=\prod_{\omega\in\Psi}X^\omega$ and $X|_\Psi=(X^\omega)_{\omega\in\Psi}$ for brevity, and similarly for maps. 

\begin{construction}[multiple deformation space]\label{multiple-deformation-space}
    For every finite set $\Omega$ and every $\Omega$-family of quasismooth closed immersions $i=(i^\omega\colon Y^\omega\to X^\omega)_{\omega\in\Omega}$ in $\cP_S$, we define the $\ps(\Omega)^\op$-diagram of closed immersions
    $$\C(i)=(\C_\Psi(i)\colon Y^\Psi\times X^{\Omega\setminus\Psi}\times0\to X^\Omega\times\bP^1)_{\Psi\subseteq\Omega}$$
    in $\cP_S$, and consider its blowup $\D(i)=\Bl(\C(i))$ (Definition \ref{definition-blowup-poset}). By Example \ref{example-product-excessive}, $\C(i)$ is excessive (Definition \ref{definition-excessive-diagram}), so by Theorem \ref{theorem-excessive-diagram-blowup-representable}, $\D(i)$ and arbitrary intersections of the strict exceptional divisors $(\E'_\Psi\D(i))_{\Psi\subseteq\Omega}$ remain in $\cP_S$. Roughly speaking, Theorem \ref{theorem-excessive-diagram-blowup-representable} exhibits $\D(i)$ as the iterated blowup of $X^\Omega\times\bP^1$ along 
    \begin{itemize}
        \item first $Y^\Omega$, 
        \item next the strict transform of $Y^{\Omega\setminus\omega}\times X^\omega$ for all $\omega\in\Omega$, 
        \item then that of $Y^{\Omega\setminus\Psi}\times X^\Psi$ for all $2$-element subsets $\Psi\subseteq\Omega$, 
        \item ...
        \item and finally that of $X^\Omega$, 
    \end{itemize}
    successively, and $\E'_\Psi\D(i)$ as the strict transform of the exceptional divisor appearing in the step that blows up $Y^\Psi\times X^{\Omega\setminus\Psi}$. 
    
    Similarly to Construction \ref{construction-gysin}, form the quotient in $\Sh_{\Zar,\qbu}(\cP_S)_*$ of $\D(i)$ by the divisors $(\E'_\Psi\D(i))_{\Psi\subsetneq\Omega}$, which we denote by $\Q(i)$, as well as that of $\E'_\Omega\D(i)$ by the divisors $(\E'_\Omega\D(i)\cap\E'_\Psi\D(i))_{\Psi\subsetneq\Omega}$, which we denote by $\E'_\Omega\Q(i)$. More precisely, form the $\ps(\Omega)$-cube given by intersecting the divisors $(\E'_\Psi\D(i))_{\Psi\in\ps(\Omega)}$, take total cofibers in all the directions except that of $\Omega\in\ps(\Omega)$ (so $2^{|\Omega|}-1$ directions), and get a map $\E'_\Omega\Q(i)\to\Q(i)$. By Theorem \ref{theorem-blowup-cofiber-comparison}, its cofiber is $X^\Omega\wedge\bP^1/0\in\Sh_{\Zar,\qbu}(\cP_S)_*$. In other words, in $\Sh_{\Zar,\qbu}(\cP_S)_*$ we have a cofiber sequence
    $$\E'_\Omega\Q(i)\to\Q(i)\to X^\Omega\wedge\bP^1/0$$
    which we call the \emph{fundamental cofiber sequence} of $i=(i^\omega\colon Y^\omega\to X^\omega)_{\omega\in\Omega}$. Note that $\E'_\Omega\D(i)$ is exactly the $\bB(\N_i)$ in \cite[\S 3]{annala-hoyois-iwasa-2023}, and $\E'_\Omega\Q(i)$ is the $\bB(\N_i)/\partial\bB(\N_i)$ there. Similarly to Construction \ref{construction-gysin}, the immersion $X^\Omega=X^\Omega\times1\to X^\Omega\times\bP^1$ as well as its nullhomotopy in $X^\Omega\wedge\bP^1/0$ induces the \emph{multiple Gysin map} 
    $$\gys(i)\colon X^\Omega\to\E'_\Omega\Q(i).$$
    
    When $\Omega$ is a singleton and there is only one immersion $i\colon Y\to X$, we have $\D(i)=\Bl_{Y\times0}(X\times\bP^1)$, $\E'_\Omega\D(i)=\bP_Y(\N_i\oplus\cO)$ and $\E'_\varnothing\D(i)=\Bl_YX$, so the $\Q(i)$ here is the $\Q(i)$ in Construction \ref{construction-gysin}, the $\E'_\Omega\Q(i)$ here is the $\Th_Y(\N_i)$ there, and the fundamental cofiber sequence here coincides with the one there. 

    Now, properties of such iterated blowups give rise to functorialities of the multiple deformation space with respect to the finite set $\Omega$: 
    \begin{enumerate}
        \item\label{multiple-deformation-space:trivial-functoriality} By Corollary \ref{corollary-blowup-cofiber-comparison-minimum}, the canonical map $\D(i)\to\D(i^\Omega)$ induces an isomorphism from the fundamental cofiber sequence of $i$ to that of $i^\Omega\colon Y^\Omega\to X^\Omega$, whose restriction $\E'_\Omega\Q(i)\to\Th(\N_{i^\Omega})$ gives the one half of the symmetric monoidality construction as in \cite[Proposition 3.6]{annala-hoyois-iwasa-2023}. More generally, for any map $\varphi\colon\Omega\to\Gamma$ of finite sets, define the $\Gamma$-family $\varphi_*i$ as
        $$(\varphi_*i)^\gamma=i^{\varphi^{-1}(\gamma)}\colon Y^{\varphi^{-1}(\gamma)}\to X^{\varphi^{-1}(\gamma)};$$
        then the canonical map $\D(i)\to\D(\varphi_*i)$ induces an isomorphism from the fundamental cofiber sequence of $i$ to that of $\varphi_*i$, compatibly with composition of the map $\varphi$. 
        \item\label{multiple-deformation-space:product-functoriality} For a map $\pi\colon\Omega\to\Xi$ of finite sets, thought of as a partition of $\Omega$ indexed by $\Xi$, by Remark \ref{functoriality-blowup-poset}(\ref{functoriality-blowup-poset:change-poset}), there is a canonical map $\D(i)\to\prod_{\xi\in\Xi}\D(i|_{\Omega_\xi})$ which takes $\E'_\Psi\D(i)$ to $\prod_{\xi\in\Xi}\E'_{\Psi_\xi}\D(i|_{\Omega_\xi})$ for $\Psi\subseteq\Omega$, where $\Omega_\xi$ denotes the preimage $\pi^{-1}(\xi)$ of $\xi$ in $\Omega$ and similarly for $\Psi_\xi$. For $\Psi=\Omega$ this map gives the other half of the symmetric monoidality construction as in \cite[Proposition 3.6]{annala-hoyois-iwasa-2023}, which accordingly identifies $\gys(i)$ with $\bigotimes_{\xi\in\Xi}\gys(i|_{\Omega_\xi})$, compatibly with refinements of partitions. 
    \end{enumerate}
\end{construction}

The following theorem is the analog of \cite[\S 3]{annala-hoyois-iwasa-2023} for Gysin maps. 

\begin{theorem}[symmetric monoidality]\label{gysin-symmetric-monoidal}
    Let $S$ be a scheme. The Gysin map is canonically symmetric monoidal as a functor taking quasismooth closed immersions in $\cP_S$ to arrows in $\PS(\cP_S)$. In other words, for a finite family of quasismooth closed immersions $(i^\omega\colon Y^\omega\to X^\omega)_{\omega\in\Omega}$ in $\cP_S$, there is a canonical commutative diagram
    $$\begin{tikzcd}
        \bigotimes_{\omega\in\Omega}X^\omega\ar[rr,"\bigotimes_{\omega\in\Omega}\gys(i^\omega)"]\ar[d,equal]&&\bigotimes_{\omega\in\Omega}\Th_{Y^\omega}(\N_{i^\omega})\ar[d,equal]\\
        X^\Omega\ar[rr,"\gys(i^\Omega)"']&&\Th_{Y^\Omega}(\N_{i^\Omega})
    \end{tikzcd}$$
    in $\PS(\cP_S)$, functorially with respect to $\Omega$. 
\end{theorem}

\begin{proof}
    This follows immediately from Construction \ref{multiple-deformation-space}. More precisely, a symmetric monoidal functor is a functor
    $$\Span(\Fin)\times[1]\to\Cat.$$
    In order to define a symmetric monoidal functor as promised, we have to specify categories and functors for objects and maps in $\Span(\Fin)\times[1]$ compatibly with composition. Among these, the only data yet to be provided is for every span
    $$\Omega'\gets\Omega\to\Xi$$
    a functor taking $\Omega'$-families of quasismooth closed immersions in $\cP_S$ to $\Xi$-families of arrows in $\PS(\cP_S)$. Clearly, the only reasonable choice for this is
    \begin{itemize}
        \item first pulling an $\Omega'$-family back to an $\Omega$-family, and
        \item next for each fiber $\Omega_\xi$ of the map $\Omega\to\Xi$, using the multiple Gysin map in Construction \ref{multiple-deformation-space} for the $\Omega_\xi$-subfamily.
    \end{itemize}
    Then (\ref{multiple-deformation-space:trivial-functoriality}) and (\ref{multiple-deformation-space:product-functoriality}) of Construction \ref{multiple-deformation-space} exactly guarantee the composition compatibility of this construction, thus defining the desired symmetric monoidal functor. 
\end{proof}

\subsection{Normalization}
In fact, the normalization theorem by the author has appeared as \cite[Proposition 4.7]{hoyois2024remarksmotivicsphere}. We include it here to emphasize that it holds in the generality of this paper. 

\begin{theorem}[normalization]\label{gysin-normalization}
    Let $S$ be a scheme. Let $Y\in\cP_S$ and $\cE$ be a vector bundle on $Y$. Let $0\colon Y\to\bV_Y(\cE)$ be the zero section. Then the Gysin map
    $$\gys(0)\colon\frac{\bV_Y(\cE)}{\bV_Y(\cE)\setminus0}\to\Th_Y(\cE)$$
    can be canonically identified with the retraction in Proposition \ref{proposition-unstable-gysin} put in $\PS(\cP_S)$. In particular, the Gysin map $\bP_Y(\cE\oplus\cO)\to\bP_Y(\cE\oplus\cO)/\bP_Y(\cE)=\Th_Y(\cE)$ of the immersion $\bP_Y(\cO)\to\bP_Y(\cE\oplus\cO)$ is the canonical quotient map. 
\end{theorem}

\begin{proof}
    Specializing the fundamental cofiber sequence in this case, we get
    $$\frac{\bP_Y(\cE\oplus\cO)}{\bP_Y(\cE)}\to\frac{\Bl_{(0,0)}(\bV_Y(\cE)\times\bP^1)}{\Bl_{(0,0)}(\bV_Y(\cE)\times0)}\to\frac{\bV_Y(\cE)\times\bP^1}{\bV_Y(\cE)\times0}$$
    and note that it splits in $\PS(\cP_S)$ canonically by the map
    $$\Bl_{(0,0)}(\bV_Y(\cE)\times\bP^1)\to\bP_Y(\cE\oplus\cO)$$
    $$(v,t)\mapsto[v:t]$$
    which maps $\Bl_{(0,0)}(\bV_Y(\cE)\times0)$ to $\bP_Y(\cE)$. This canonically identifies the Gysin map $\bV_Y(\cE)\to\Th_Y(\cE)$ with the inclusion $\bV_Y(\cE)\to\bP_Y(\cE\oplus\cO)$ (which is $v\mapsto[v:1]$ in coordinates) composed with the quotient map. Now it remains to identify the nullhomotopy of $\bV_Y(\cE)\setminus0\to\bV_Y(\cE)\to\Th_Y(\cE)$ given by the Gysin map with the one in Proposition \ref{proposition-unstable-gysin}. By the above splitting, the former is given by
    $$(\bV_Y(\cE)\setminus0)\times\bP^1\to\bP_Y(\cE\oplus\cO)$$
    $$(v,t)\mapsto[v:t]$$
    as a $\bP^1$-homotopy from the inclusion $\bV_Y(\cE)\setminus0\to\bP_Y(\cE\oplus\cO)$ to the composition $\bV_Y(\cE)\setminus0\to\bP_Y(\cE)\to\bP_Y(\cE\oplus\cO)$; since this $\bP^1$-homotopy canonically factors through $\Bl_{\bP_Y(\cO)}\bP_Y(\cE\oplus\cO)$ with a commutative diagram
    $$\begin{tikzcd}
        (\bV_Y(\cE)\setminus0)\times\bA^1\ar[r]\ar[d]&\bP_Y(\cE\oplus\cO)\setminus\bP_Y(\cO)=\bV_{\bP_Y(\cE)}(\cO(-1))\ar[d]\\
        (\bV_Y(\cE)\setminus0)\times\bP^1\ar[r]&\Bl_{\bP_Y(\cO)}\bP_Y(\cE\oplus\cO)=\bP_{\bP_Y(\cE)}(\cO(-1)\oplus\cO)
    \end{tikzcd}$$
    compatible with the proof of Proposition \ref{proposition-unstable-gysin}, we see that the two nullhomotopies are canonically equal. 
\end{proof}

\begin{corollary}\label{gysin-normalization-corollary}
    Let $S$ be a scheme. Let $Z\in\cP_S$ and $\cE,\cF$ be vector bundles on $Z$. Then the Gysin map of the immersion $\bP_Z(\cF)\to\bP_Z(\cE\oplus\cF)$ can be expressed as
    \begin{align*}
        \bP_Z(\cE\oplus\cF)&\to\frac{\bP_Z(\cE\oplus\cF)}{\bP_Z(\cE)}=\frac{\Bl_{\bP_Z(\cE)}\bP_Z(\cE\oplus\cF)}{\E_{\bP_Z(\cE)}\bP_Z(\cE\oplus\cF)}=\frac{\bP_{\bP_Z(\cF)}(\cE(-1)\oplus\cO)}{\bP_{\bP_Z(\cF)}(\cE(-1))}
        \\&=\Th_{\bP_Z(\cF)}(\cE(-1)),
    \end{align*}
    where we note that $\N_{\bP_Z(\cF)/\bP_Z(\cE\oplus\cF)}=\cE(-1)$. 
    
    In particular, taking quotient of the Gysin map of $\bP_Z(\cF\oplus\cO)\to\bP_Z(\cE\oplus\cF\oplus\cO)$ by that of $\bP_Z(\cF)\to\bP_Z(\cE\oplus\cF)$, the resulting map
    $$\gys(\Th_Z(\cF)\to\Th_Z(\cE\oplus\cF))\colon\Th_Z(\cE\oplus\cF)\to\Th_{\Th_Z(\cF)}(\cE(-1))$$
    is an isomorphism.
\end{corollary}

\begin{proof}
    Since $\bP_Z(\cE)$ and $\bP_Z(\cF)$ are disjoint in $\bP_Z(\cE\oplus\cF)$, the Gysin map in question factors through the quotient $\frac{\bP_Z(\cE\oplus\cF)}{\bP_Z(\cE)}=\frac{\Bl_{\bP_Z(\cE)}\bP_Z(\cE\oplus\cF)}{\E_{\bP_Z(\cE)}\bP_Z(\cE\oplus\cF)}$ and equals the Gysin map of $\bP_Z(\cF)\to\Bl_{\bP_Z(\cE)}\bP_Z(\cE\oplus\cF)$ factoring through the same quotient. Now the corollary follows from Theorem \ref{gysin-normalization} by the identification $\Bl_{\bP_Z(\cE)}\bP_Z(\cE\oplus\cF)=\bP_{\bP_Z(\cF)}(\cE(-1)\oplus\cO)$ that respects the $\bP_Z(\cF)$ immersed in both sides. 
\end{proof}

\begin{corollary}\label{thom-splitting}
    Let $S$ be a scheme and let $Z\in\cP_S$. Let
    $$0\to\cE\to\cG\to\cF\to0$$
    be a short exact sequence of finite locally free $\cO_Z$-modules. Note that in this situation we still have $\N_{\bP_Z(\cF)/\bP_Z(\cG)}=\cE(-1)$. Going through the equalities in Corollary \ref{gysin-normalization-corollary} the other way, we get a map
    $$\gys(\bP_Z(\cF)\to\bP_Z(\cG))\colon\bP_Z(\cG)\to\Th_{\bP_Z(\cF)}(\cE(-1))=\frac{\bP_Z(\cE\oplus\cF)}{\bP_Z(\cE)}.$$
    Doing the above construction also for $\cG\oplus\cO\to\cF\oplus\cO$ and taking quotient, we get
    $$\gys(\Th_Z(\cF)\to\Th_Z(\cG))\colon\Th_Z(\cG)\to\Th_{\Th_Z(\cF)}(\cE(-1))=\Th_Z(\cE\oplus\cF).$$
    This map is an isomorphism in $\PS(\cP_S)$. 
\end{corollary}

\begin{proof}
    One can assume $Z=S$ by $f_\sharp$-functoriality. Isomorphy in $\PS(\cP_S)$ can be checked Zariski locally on $S$, which means that we can assume $\cG=\cE\oplus\cF$, in which case it follows from Corollary \ref{gysin-normalization-corollary}. 
\end{proof}

\begin{remark}\label{thom-splitting-symmetric-monoidal}
    Note that the map $\gys(\Th_Z(\cF)\to\Th_Z(\cG))$ in Corollary \ref{thom-splitting} and the map $\gys(\Th_Z(\cF)\to\Th_Z(\cE\oplus\cF))$ in Corollary \ref{gysin-normalization-corollary} can be made symmetric monoidal with respect to the direct sum, because the construction of $(\bB,\partial\bB)$ in \cite[\S 3]{annala-hoyois-iwasa-2023} is functorial and cartesian in surjections of locally free $\cO_Z$-modules. 
\end{remark}

The following proposition compares the Thom splitting isomorphism provided by Corollary \ref{thom-splitting} with the one coming from invertibility of Thom objects in $\MS$. 

\begin{proposition}[comparing Thom splittings]\label{thom-splitting-comparison}
    In $\MS(\cP_S)$, the isomorphism of Corollary \ref{thom-splitting} coincides with the one in the beginning of \cite[\S 7]{annala-hoyois-iwasa-2023}. 
\end{proposition}

\begin{proof}
    In fact, the latter isomorphism comes from the fact that the $\K$-theory, as a Zariski sheaf, can be defined in two ways: 
    \begin{enumerate}
        \item\label{thom-splitting-comparison:quillen} Quillen $+$-construction, i.e.\ $\K=(\vect,\oplus)^\gp$ as the group completion of the $\EE_\infty$-monoid of vector bundles with respect to the direct sum;
        \item\label{thom-splitting-comparison:waldhausen} Waldhausen $\ws$-construction, i.e.\ $\K=\Omega|\ws_\bullet\Vect|$ as the loop space of the geometric realization of the simplicial anima $\ws_\bullet\Vect$ defined as in \cite[Lecture 18, Definition 9]{lurie-281}. In short, a point in $\ws_n\Vect$ is a vector bundle along with a length-$n$ filtration whose graded pieces are still vector bundles; 
    \end{enumerate}
    by (\ref{thom-splitting-comparison:quillen}), since $\Th_Z\colon(\vect_Z,\oplus)\to(\MS(\cP_Z),\otimes_Z)$ lands in invertible objects, it factors through $\K$; then by (\ref{thom-splitting-comparison:waldhausen}), the $2$-cell
    $$(0\to\cE\to\cG\to\cF\to0)\in\ws_2\Vect_Z$$
    witnesses an equality $[\cE]+[\cF]=[\cG]$ in $\K=\Omega|\ws_\bullet\Vect_Z|$, so we get an isomorphism $\Th_Z(\cE)\otimes\Th_Z(\cF)=\Th_Z(\cG)$ accordingly. Composing it with the isomorphism of Corollary \ref{thom-splitting} gives a functor
    $$\ws_2\Vect_Z\to\MS(\cP_Z)^{\Sone}:=\{(X,f)\mid X\in\MS(\cP_Z),\,f\colon X\xrightarrow{\sim}X\}$$
    \begin{align*}
        (0\to\cE\to\cG\to\cF\to0)&\mapsto\\
        (\Th_Z(\cG)&\xrightarrow{\text{Corollary \ref{thom-splitting}}}\Th_Z(\cE\oplus\cF)\xrightarrow{\K\text{-theory}}\Th_Z(\cG))
    \end{align*}
    which is symmetric monoidal under the direct sum on the left and the pointwise tensor product on the right by Remark \ref{thom-splitting-symmetric-monoidal}. Since Thom spaces are invertible in $\MS(\cP_Z)$, the functor factors through the group completion of $\ws_2\Vect_Z$. Now, from the proof of the equivalence of the two ways to define the $\K$-theory (more precisely, \cite[Lecture 18, Theorem 10]{lurie-281}), we see that the map
    $$(\vect^2,\oplus)\to(\ws_2\Vect,\oplus)$$
    $$(\cE,\cF)\mapsto(0\to\cE\to\cE\oplus\cF\to\cF\to0)$$
    of Zariski sheaves of $\EE_\infty$-monoids becomes an isomorphism after group completion. Therefore, to prove that the above functor $\ws_2\Vect_Z\to\MS(\cP_Z)^{\Sone}$ factors through $\MS(\cP_Z)\to\MS(\cP_Z)^{\Sone}$, $X\mapsto(X,\id_X)$, it suffices to do so after precomposed with the map $\vect^2\to\ws_2\Vect$ above, but this is clear because the isomorphism in Corollary \ref{thom-splitting} is defined as the composite of 
    $$\gys(\Th_Z(\cF)\to\Th_Z(\cG))\colon\Th_Z(\cG)\to\Th_{\Th_Z(\cF)}(\cE(-1))$$
    and the inverse of
    $$\gys(\Th_Z(\cF)\to\Th_Z(\cE\oplus\cF))\colon\Th_Z(\cE\oplus\cF)\to\Th_{\Th_Z(\cF)}(\cE(-1))$$
    which is manifestly the identity for $\cG=\cE\oplus\cF$. 
\end{proof}

\subsection{Composition}
The following is the composition-of-two-immersion version of Construction \ref{construction-gysin}, and will be the main ingredient of the composition compatibility of Gysin maps. Due to its complexity, we advice the reader to skip it first. 

\begin{construction}[composite deformation space]\label{composite-deformation-space}
    Let $Z\to Y\to X$ be a chain of two quasismooth closed immersions. Define the following schemes: 
    $$\D(Z\to Y\to X)=\Bl'\left(\begin{tikzcd}
        Z\times0\times0\ar[r]\ar[d]&Z\times0\times\bP^1\ar[d]\\
        Y\times\bP^1\times0\ar[r]&X\times\bP^1\times\bP^1
    \end{tikzcd}\right),$$
    $$\partial_Y\D(Z\to Y\to X)=\Bl'\left(\begin{tikzcd}
        Z\times0\times0\ar[r]\ar[d]&Z\times0\times0\ar[d]\\
        Y\times\bP^1\times0\ar[r]&X\times\bP^1\times0
    \end{tikzcd}\right),$$
    $$\partial_Z\D(Z\to Y\to X)=\Bl'\left(\begin{tikzcd}
        Z\times0\times0\ar[r]\ar[d]&Z\times0\times\bP^1\ar[d]\\
        Y\times0\times0\ar[r]&X\times0\times\bP^1
    \end{tikzcd}\right),$$
    $$\partial_{Y,Z}\D(Z\to Y\to X)=\Bl'\left(\begin{tikzcd}
        Z\times0\times0\ar[r]\ar[d]&Z\times0\times0\ar[d]\\
        Y\times0\times0\ar[r]&X\times0\times0
    \end{tikzcd}\right),$$
    where $\Bl'$ is the pushout-blowup in Definition \ref{definition-pushout-blowup}, with the commutative square
    $$\begin{tikzcd}
        \partial_{Y,Z}\D(Z\to Y\to X)\ar[r]\ar[d]&\partial_Z\D(Z\to Y\to X)\ar[d]\\
        \partial_Y\D(Z\to Y\to X)\ar[r]&\D(Z\to Y\to X)
    \end{tikzcd}$$
    where the four arrows are divisors. When $Z\to Y\to X$ is understood, we omit it from the notation. By Definition \ref{definition-pushout-blowup}, these four blowups come with three strict exceptional divisors each. We denote $\E'_Y\D=\E'_{Y\times\bP^1\times0}\D$, $\E'_Z\D=\E'_{Z\times0\times\bP^1}\D$, and $\E'_{Y,Z}\D=\E'_{Z\times0\times0}\D$, and similarly for $\partial_Y\D$, $\partial_Z\D$, and $\partial_{Y,Z}\D$. By Theorem \ref{pushout-blowup-description}, it is easy to see that $\E'_Z\partial_Y\D=\E'_Z\partial_{Y,Z}\D=\varnothing$. 

    Now let $S$ be a scheme and suppose that $Z\to Y\to X$ is in $\cP_S$. Define objects $\Q,\E'_Y\Q,\E'_Z\Q,\E'_Y\Q\cap\E'_Z\Q$ as the total cofibers of
    $$\begin{tikzcd}
        \dfrac{\partial_{Y,Z}\D}{\E'_{Y,Z}\partial_{Y,Z}\D}\ar[r]\ar[d]&\dfrac{\partial_Z\D}{\E'_{Y,Z}\partial_Z\D}\ar[d]\\
        \dfrac{\partial_Y\D}{\E'_{Y,Z}\partial_Y\D}\ar[r]&\dfrac{\D}{\E'_{Y,Z}\D}
    \end{tikzcd}$$
    $$\begin{tikzcd}
        \dfrac{\E'_Y\partial_{Y,Z}\D}{\E'_Y\partial_{Y,Z}\D\cap\E'_{Y,Z}\partial_{Y,Z}\D}\ar[r]\ar[d]&\dfrac{\E'_Y\partial_Z\D}{\E'_Y\partial_Z\D\cap\E'_{Y,Z}\partial_Z\D}\ar[d]\\
        \dfrac{\E'_Y\partial_Y\D}{\E'_Y\partial_Y\D\cap\E'_{Y,Z}\partial_Y\D}\ar[r]&\dfrac{\E'_Y\D}{\E'_Y\D\cap\E'_{Y,Z}\D}
    \end{tikzcd}$$
    $$\begin{tikzcd}
        \dfrac{\E'_Z\partial_{Y,Z}\D}{\E'_Z\partial_{Y,Z}\D\cap\E'_{Y,Z}\partial_{Y,Z}\D}\ar[r]\ar[d]&\dfrac{\E'_Z\partial_Z\D}{\E'_Z\partial_Z\D\cap\E'_{Y,Z}\partial_Z\D}\ar[d]\\
        \dfrac{\E'_Z\partial_Y\D}{\E'_Z\partial_Y\D\cap\E'_{Y,Z}\partial_Y\D}\ar[r]&\dfrac{\E'_Z\D}{\E'_Z\D\cap\E'_{Y,Z}\D}
    \end{tikzcd}$$
    $$\begin{tikzcd}
        \dfrac{\E'_Y\partial_{Y,Z}\D\cap\E'_Z\partial_{Y,Z}\D}{\E'_Y\partial_{Y,Z}\D\cap\E'_Z\partial_{Y,Z}\D\cap\E'_{Y,Z}\partial_{Y,Z}\D}\ar[r]\ar[d]&\dfrac{\E'_Y\partial_Z\D\cap\E'_Z\partial_Z\D}{\E'_Y\partial_Z\D\cap\E'_Z\partial_Z\D\cap\E'_{Y,Z}\partial_Z\D}\ar[d]\\
        \dfrac{\E'_Y\partial_Y\D\cap\E'_Z\partial_Y\D}{\E'_Y\partial_Y\D\cap\E'_Z\partial_Y\D\cap\E'_{Y,Z}\partial_Y\D}\ar[r]&\dfrac{\E'_Y\D\cap\E'_Z\D}{\E'_Y\D\cap\E'_Z\D\cap\E'_{Y,Z}\D}
    \end{tikzcd}$$
    respectively, taken in $\Sh_{\Zar,\qbu}(\cP_S)_*$, where $\cap$ for divisors means fiber products. Then by Theorem \ref{pushout-blowup-description}, there is a canonical commutative diagram
    $$\begin{tikzcd}
        \E'_Y\Q\cap\E'_Z\Q\ar[r]\ar[d]&\E'_Z\Q\ar[r]\ar[d]&\Th(\N_{Z/X})\wedge\bP^1/0\ar[d]\\
        \E'_Y\Q\ar[r]\ar[d]&\Q\ar[r]\ar[d]&\Q(Z\to X)\wedge\bP^1/0\ar[d]\\
        \Th(\N_{Y/X})\wedge\bP^1/0\ar[r]&\Q(Y\to X)\wedge\bP^1/0\ar[r]&X\wedge\bP^1/0\wedge\bP^1/0
    \end{tikzcd}$$
    in $\Sh_{\Zar,\qbu}(\cP_S)_*$, whose rows and columns can be described as follow:
    \begin{enumerate}
        \item\label{composite-deformation-space:middle-row} Its middle row is the fundamental cofiber sequence of 
        $$\Bl_{Z\times0}(Y\times\bP^1)\to\Bl_{Z\times0}(X\times\bP^1)$$
        quotiented by that of 
        $$\Bl_{Z\times0}(Y\times0)\to\Bl_{Z\times0}(X\times0).$$
        \item\label{composite-deformation-space:upper-row} Its upper row is the fundamental cofiber sequence of 
        $$\E_{Z\times0}(Y\times\bP^1)\to\E_{Z\times0}(X\times\bP^1)$$ 
        quotiented by that of 
        $$\E_{Z\times0}(Y\times0)\to\E_{Z\times0}(X\times0).$$
        \item\label{composite-deformation-space:lower-row} Its lower row is the fundamental cofiber sequence of 
        $$Y\times\bP^1\to X\times\bP^1$$
        quotiented by that of 
        $$Y\times0\to X\times0.$$
        \item\label{composite-deformation-space:middle-column} Its middle column is the fundamental cofiber sequence of 
        $$Z\times\bP^1=\Bl_{Z\times0}(Z\times\bP^1)\to\Bl_{Y\times0}(X\times\bP^1)$$
        quotiented by that of 
        $$\varnothing=\Bl_{Z\times0}(Z\times0)\to\Bl_{Y\times0}(X\times0).$$
        \item\label{composite-deformation-space:left-column} Its left column is the fundamental cofiber sequence of 
        $$Z\times0=\E_{Z\times0}(Z\times\bP^1)\to\E_{Y\times0}(X\times\bP^1)$$ 
        quotiented by that of 
        $$\varnothing=\E_{Z\times0}(Z\times0)\to\E_{Y\times0}(X\times0).$$
        \item\label{composite-deformation-space:right-column} Its right column is the fundamental cofiber sequence of 
        $$Z\times\bP^1\to X\times\bP^1$$ 
        quotiented by that of 
        $$Z\times0\to X\times0.$$
    \end{enumerate}
    Note that (\ref{composite-deformation-space:upper-row}) and (\ref{composite-deformation-space:left-column}) give rise to isomorphisms
    $$\E'_Y\Q\cap\E'_Z\Q=\Th_{\Th_Z(\N_{Z/Y})}((\N_{Y/X}|_Z)(-1))$$
    and
    $$\E'_Y\Q\cap\E'_Z\Q=\Th_Z(\N_{Y/X}|_Z\oplus\N_{Z/Y})$$
    respectively, which compose to the equality at the end of Corollary \ref{thom-splitting}. 
\end{construction}

The following lemma is an immediate consequence of the above construction. 

\begin{lemma}\label{semi-composition-gysin}
    Let $S$ be a scheme and let $Z\xrightarrow{j}Y\xrightarrow{i}X$ be a chain of two quasismooth closed immersions in $\cP_S$. Then there is a canonical commutation of the square
    $$\begin{tikzcd}
        X\ar[rrrr,"\gys(ij)"]\ar[d,"\gys(i)"']&&&&\Th_Z(\N_{ij})\ar[d,equal,"\gys(\Th_Z(\N_{Z/i}))"]\\
        \Th_Y(\N_i)\ar[rrrr,"\gys(0\colon Z\to\Th_Y(\N_i))"']&&&&\Th_Z(j^*\N_i\oplus\N_j)
    \end{tikzcd}$$
    in $\PS(\cP_S)$, where 
    \begin{enumerate}
        \item\label{semi-composition-gysin:two-step} $\gys(\Th_Z(\N_{Z/i}))$ means the Gysin map of $\bP_Z(\N_j\oplus\cO)\to\bP_Z(\N_{ij}\oplus\cO)$ quotiented by that of $\bP_Z(\N_j)\to\bP_Z(\N_{ij})$ as in Corollary \ref{thom-splitting} applied to
        $$0\to j^*\N_i\to\N_{ij}\to\N_j\to0$$
        and hence is an isomorphism; 
        \item\label{semi-composition-gysin:one-step} $\gys(0\colon Z\to\Th_Y(\N_i))$ means the Gysin map of the composition $Z\to Y=\bP_Y(\cO)\to\bP_Y(\N_i\oplus\cO)$ quotiented by that of $\varnothing\to\bP_Y(\N_i)$. 
    \end{enumerate}
\end{lemma}

\begin{proof}
    This follows immediately from Construction \ref{composite-deformation-space} using the same method as in Construction \ref{construction-gysin}, noting that the descriptions (\ref{semi-composition-gysin:two-step}) and (\ref{semi-composition-gysin:one-step}) above align exactly with (\ref{composite-deformation-space:upper-row}) and (\ref{composite-deformation-space:left-column}) in Construction \ref{composite-deformation-space}, respectively. 
\end{proof}

\begin{theorem}[composition]\label{theorem-composition-gysin}
    Let $S$ be a scheme and let $Z\xrightarrow{j}Y\xrightarrow{i}X$ be a chain of two quasismooth closed immersions in $\cP_S$. Then there is a canonical commutation of the square
    $$\begin{tikzcd}
        X\ar[rrr,"\gys(ij)"]\ar[dd,"\gys(i)"']&&&\Th_Z(\N_{ij})\ar[d,equal,"\gys(\Th_Z(\N_{Z/i}))"]\\
        &&&\Th_Z(j^*\N_i\oplus\N_j)\ar[d,equal]\\
        \Th_Y(\N_i)\ar[rrr,"\gys(\Th_j(\N_i))"']&&&\Th_Z(j^*\N_i)\otimes_Z\Th_Z(\N_j)
    \end{tikzcd}$$
    in $\PS(\cP_S)$, where 
    \begin{enumerate}
        \item $\gys(\Th_Z(\N_{Z/i}))$ means the Gysin map of $\bP_Z(\N_j\oplus\cO)\to\bP_Z(\N_{ij}\oplus\cO)$ quotiented by that of $\bP_Z(\N_j)\to\bP_Z(\N_{ij})$ as in Corollary \ref{thom-splitting} applied to
        $$0\to j^*\N_i\to\N_{ij}\to\N_j\to0$$
        and hence is an isomorphism; 
        \item $\gys(\Th_j(\N_i))$ means the Gysin map of $\bP_Z(j^*\N_i\oplus\cO)\to\bP_Y(\N_i\oplus\cO)$ quotiented by that of $\bP_Z(j^*\N_i)\to\bP_Y(\N_i)$. 
    \end{enumerate}
\end{theorem}

\begin{proof}
    Lemma \ref{semi-composition-gysin} reduces the question to identifying $\gys(0\colon Z\to\Th_Y(\N_i))$ and $\gys(\Th_j(\N_i))$, up to the symmetric monoidality of Definition \ref{definition-thom-space}. For this, apply Lemma \ref{semi-composition-gysin} again, now on $Z\xrightarrow{0}\Th_Z(j^*\N_i)\to\Th_Y(\N_i)$. Namely, apply Lemma \ref{semi-composition-gysin} on $Z=\bP_Z(\cO)\to\bP_Z(j^*\N_i\oplus\cO)\to\bP_Y(\N_i\oplus\cO)$ and $\varnothing\to\bP_Z(j^*\N_i)\to\bP_Y(\N_i)$, and take quotient. This gives us the commutative square
    $$\begin{tikzcd}
        \Th_Y(\N_i)\ar[rrr,"\gys(0\colon Z\to\Th_Y(\N_i))"]\ar[d,"\gys(\Th_j(\N_i))"']&&&\Th_Z(j^*\N_i\oplus\N_j)\ar[d,equal]\\
        \Th_Z(j^*\N_i)\otimes_Z\Th_Z(\N_j)\ar[rrr]&&&\Th_Z(j^*\N_i\oplus\N_j)
    \end{tikzcd}$$
    where
    \begin{enumerate}
        \item the right vertical isomorphism is $\gys(\Th_Z(j^*\N_i)\to\Th_Z(j^*\N_i\oplus\N_j))$ which by Corollary \ref{gysin-normalization-corollary} is the identity; 
        \item the lower horizontal arrow is $\gys(0\colon Z\to\Th_Z(j^*\N_i)\otimes_Z\Th_Z(\N_j))$ which by Theorem \ref{gysin-normalization} and Remark \ref{thom-space-additive-alternative} is the symmetric monoidal structure in Definition \ref{definition-thom-space}. 
    \end{enumerate}
    This finishes the proof. 
\end{proof}

\section{Comparisons and applications}

In this section, we compare our Gysin map with existing constructions in various special cases, and give some applications of our Gysin map. As in the previous section, we fix a motivic setup $\cP$ and a base scheme $S$. For $X\in\cP_S$ denote its structure map by $a_X\colon X\to S$. 

\subsection{Gysin maps for various cohomology theories}
In this subsection, we use results in \cite{annala-iwasa-2022} to represent various cohomology theories as objects in $\MS$, and thus construct Gysin maps for them. 

\begin{construction}[representing cohomology theories as motivic spectra]\label{represent-cohomology-in-ms}
    Recall that $\MS(\cP_S)$ is defined by smash-inverting the compact object $\bP^1/\infty$ in the compactly generated category $\Sh_{\Zar,\qbu}(\cP_S;\Sp)$ with compact objects closed under smash products. Therefore, $\MS(\cP_S)$ is also compactly generated with compact objects closed under tensor products, and the full subcategory $\MS(\cP_S)^\omega$ of compact objects is also given by smash-inverting $\bP^1/\infty\in\Sh_{\Zar,\qbu}(\cP_S;\Sp)^\omega$, in the world of small idempotent-complete stably symmetric monoidal categories. 

    Let $\fin$ denote the anima of finite sets and $\bullet$ denote the singleton. View $\fin$ as a symmetric monoidal category under the disjoint union, so it is the free symmetric monoidal category in the object $\bullet$, and thus $\PSh(\fin;\Sp)$ is the free stable presentably symmetric monoidal category in the object $\bullet$. Let $\cC$ be a $\PSh(\fin;\Sp)[\bullet^{-1}]$-module, and denote by $\langle-\Phi\rangle$ the action of the finite set $\Phi$. In other words, $\cC$ is a stable presentable category equipped with an automorphism $\langle-1\rangle$ commuting with itself. 
    
    Let $F\colon\cP_S^\op\to\cC$ be a functor that we think of as a cohomology theory. By abuse of notation, denote also by $F$ the functor $\PSh(\cP_S)^\op\to\cC$ given by right Kan extension and the functor $\PSh(\cP_S)_*^\op\to\cC$ given by the reduced cohomology. Assume: 
    \begin{enumerate}
        \item\label{represent-cohomology-in-ms:zariski} $F$ is a Zariski sheaf;
        \item\label{represent-cohomology-in-ms:qbu} $F$ satisfies quasismooth blowup excision;
        \item\label{represent-cohomology-in-ms:tate-twist} there are identifications
        $$F(\Sigma_{\bP^1}^\Phi X)=F(X)\langle-\Phi\rangle$$
        functorially for $X\in\cP_S^\op$ and the finite set $\Phi$ in automorphisms and disjoint unions of finite sets, where $\Sigma_{\bP^1}^\Phi$ means smashing with $\Phi$-many $\bP^1/\infty$.
    \end{enumerate}

    A notable special case of the above situation is when $\cC$ is presentably symmetric monoidal, $F$ is lax symmetric monoidal, $F(S)=1_\cC$, $F(\bP^1/\infty)$ is invertible, and $F(\Sigma_{\bP^1}X)=F(\bP^1/\infty)\otimes F(X)$. In this case $\langle-1\rangle$ is just tensoring with $F(\bP^1/\infty)$. 
    
    Then I claim that canonically $F$ extends to an exact functor $(\MS(\cP_S)^\omega)^\op\to\cC$ and in the special case above this extension is lax symmetric monoidal. In other words, $F$ can be represented by an object $F\in\MS(\cP_S)\otimes\cC$, and in the special case above by an object $F\in\CAlg(\MS(\cP_S)\otimes\cC)$. Moreover, $F^{(\bP^1/\infty)^{\wedge\Phi}}=F\langle-\Phi\rangle$ functorially in $\Phi\in\fin$, identifying the $\MS(\cP_S)$-twist of $F$ with the $\cC$-twist of $F$. 
    
    Indeed, (\ref{represent-cohomology-in-ms:zariski}) and (\ref{represent-cohomology-in-ms:qbu}) extend $F$ to an exact functor $(\Sh_{\Zar,\qbu}(\cP_S;\Sp)^\omega)^\op\to\cC$; then (\ref{represent-cohomology-in-ms:tate-twist}) gives $F$ a $\PSh(\fin;\Sp)^\omega$-linear structure (in the symmetric monoidal category of idempotent-complete stable categories), and thus extends $F$ further to a functor $(\Sh_{\Zar,\qbu}(\cP_S;\Sp)^\omega)^\op[(\bP^1/\infty)^{-1}]=(\MS(\cP_S)^\omega)^\op\to\cC$ by extension of scalars, since $\bullet\in\fin$ acts on $\cC$ invertibly.  Lax symmetric monoidality in the special case then follows easily, for example by the fact that extension of scalars is a symmetric monoidal $2$-functor. The final claim follows from (\ref{represent-cohomology-in-ms:tate-twist}) immediately. 
\end{construction}

The author thanks Marc Hoyois for teaching him the above argument and even the current formulation of the above construction. 

\begin{remark}[orienting cohomology theories]\label{orient-cohomology-theories}
    Keep the notation and the assumptions of Construction \ref{represent-cohomology-in-ms}. By \cite[Lemma 3.1.4]{annala-iwasa-2022}, an orientation of $F$ as in \cite[Definition 3.1.2]{annala-iwasa-2022} is a section of the $\fin$-linear natural transformation
    $$F(\B\Gm\times -)\to F(\bP^1\times -)$$
    induced by $\cO(1)\colon\bP^1\to\B\Gm$. In the special case mentioned in Construction \ref{represent-cohomology-in-ms} where $\cC$ is presentably symmetric monoidal and $F$ is lax symmetric monoidal, again by \cite[Lemma 3.1.4]{annala-iwasa-2022}, a section of the map
    $$F(\B\Gm)\xrightarrow{F(\cO(1))}F(\bP^1)\to\tilde{F}(\bP^1):=\cofib(F(S)\to F(\bP^1))$$
    suffices, and in fact gives a linear orientation as in \cite[Remark 3.1.5]{annala-iwasa-2022}. 

    It is sometimes convenient to note that, since $\B\Gm=\colim_r\bP^r$ in $\MS(\cP_S)$ by \cite[Theorem 5.3]{annala-hoyois-iwasa-2023}, an orientation of $F$ is equivalently a compatible family of sections of the $\fin$-linear natural transformations 
    $$(F(\bP^r\times -)\to F(\bP^1\times -))_{r\in\ZZ_+},$$ 
    and in the above special case a compatible family of sections of the maps 
    $$(F(\bP^r)\to F(\bP^1))_{r\in\ZZ_+}.$$
\end{remark}

\begin{example}[localizing invariants]\label{represent-localizing-invariants}
    Take our motivic setup $\cP$ to be all maps. Let $A$ be an accessible localizing invariant on $\Perf(S)$-linear categories valued in a presentable stable category $\cC$. By \cite[Theorems A and B]{khan2020algebraicktheoryofquasismoothblowupsandcdhdescent} and the fact that localizing invariants satisfy Nisnevich descent, Construction \ref{represent-cohomology-in-ms} applies to $A(\Perf(-))$ with $\fin$ acting on $\cC$ trivially, and hence gives an object $A_\MS\in\MS(\cP_S)\otimes\cC$ with $A_\MS^{\bP^1/\infty}=A_\MS$ (cf.\ \cite[\S 5.2]{annala-iwasa-2022}). \cite[Theorem B]{khan2020algebraicktheoryofquasismoothblowupsandcdhdescent} also gives compatible identifications $A_\MS^{\bP^r}=\bigoplus_{k=0}^rA_\MS$ and thus gives $A_\MS$ an orientation by Remark \ref{orient-cohomology-theories}. 
    
    This gives a Gysin map $A(\Perf(Y))\to A(\Perf(S))$ for every quasismooth closed immersion $i\colon Y\to S$, and Theorem \ref{theorem-oriented-comparision} below will show that it coincides with the obvious map induced by $i_*\colon\Perf(Y)\to\Perf(S)$.  
\end{example}

In practice we will mostly use Construction \ref{represent-cohomology-in-ms} in the following form: 

\begin{construction}[representing cohomology theories from $4$-functor formalisms]\label{represent-cohomology-4ff}
    Let $\cD\colon\cP_S^\op\to\CAlg(\PrLst)$ be a presheaf of stable presentably symmetric monoidal categories and let $A$ be an object of $\cD(S)$. For $X\in\cP_S$ write $A_X=a_X^*A$. Assume:
    \begin{enumerate}
        \item\label{represent-cohomology-4ff:zariski} $\cD$ is a Zariski sheaf; 
        \item\label{represent-cohomology-4ff:qbu} $\cD$ satisfies quasismooth blowup excision, in the sense that for every quasismooth blowup square in $\cP_S$
        $$\begin{tikzcd}
            \E_YX\ar[r]\ar[d]\ar[rd,"e" description]&\Bl_YX\ar[d,"b"]\\
            Y\ar[r,"i"']&X
        \end{tikzcd}$$
        we have $A_X=i_*i^*A_X\times_{e_*e^*A_X}b_*b^*A_X$;
        \item\label{represent-cohomology-4ff:tate-twist} for every $X\in\cP_S$ and $C\in\cD(X)$, denote $C\langle-1\rangle=\cofib(C\to p_*p^*C)$ where $p\colon\bP^1_X\to X$ is the projection; then $A_X\langle-1\rangle=A_X\otimes1_X\langle-1\rangle$, $1_X\langle-1\rangle=a_X^*1_S\langle-1\rangle$, and $1_S\langle-1\rangle$ is invertible in $\cD(S)$. 
    \end{enumerate}
    Also assume the above for $A=1_S$, which actually implies the case for all dualizable $A$ by abstract nonsense. For $X\in\cP_S$ write
    $$F_A(X)=(a_X)_*A_X.$$
    Then the above assumptions guarantee that Construction \ref{represent-cohomology-in-ms} applies to the functor $F_A\colon\cP_S^\op\to\cD(S)$ with the $\fin$-action given by tensoring with $1_S\langle-1\rangle^{\otimes\Phi}$ for $\Phi\in\fin$, and hence gives an object $A_\MS\in\MS(\cP_S)\otimes\cD(S)$ with $A_\MS^{\bP^1/\infty}=A_\MS\langle-1\rangle$. 
    
    Moreover, the external tensor product makes $A\mapsto F_A$ lax symmetric monoidal under the Day convolution, and thus makes $A\mapsto A_\MS$ lax symmetric monoidal wherever it is defined; in particular, $(1_S)_\MS$ is a commutative algebra in $\MS(\cP_S)\otimes\cD(S)$ and $A_\MS$ is its module. 
\end{construction}

\begin{remark}[orienting cohomology theories from $4$-functor formalisms]\label{orient-cohomology-4ff}
    Keep the notation and the assumptions of Construction \ref{represent-cohomology-4ff}. By Remark \ref{orient-cohomology-theories}, in order to linearly orient $(1_S)_\MS\in\CAlg(\MS(\cP_S)\otimes\cD(S))$, it suffices to give a section of
    $$F_{1_S}(\B\Gm)=\lim_{(X,\cL)\in{(\cP_S)_{/\B\Gm}}^\op}F_{1_S}(X)\xrightarrow{(\bP^1,\cO(1))}F_{1_S}(\bP^1)\to\tilde{F}_{1_S}(\bP^1)=1_S\langle-1\rangle,$$
    or in other words compatible maps
    $$(c_1(\cL)\colon1_S\langle-1\rangle\to(a_X)_*1_X)_{(X,\cL)\in(\cP_S)_{/\B\Gm}}$$
    which, when $(X,\cL)=(\bP^1,\cO(1))$, is a section of the natural map $(a_{\bP^1})_*1_{\bP^1}\to\cofib(1_S\to(a_{\bP^1})_*1_{\bP^1})=1_S\langle-1\rangle$. By \cite[Remark 3.1.6]{annala-iwasa-2022}, such data orient not only $(1_S)_\MS$ but also $A_\MS$ wherever $A$ satisfies the assumptions of Construction \ref{represent-cohomology-4ff}. 
\end{remark}

\begin{construction}[Gysin map with coefficients]\label{gysin-general-coefficient}
    Keep the notation and the assumptions of Construction \ref{represent-cohomology-4ff}. For $Y\in\cP_S$, $B\in\cD(Y)$, and $\cE$ a finite locally free $\cO_Y$-module, denote $B\langle-\cE\rangle=\fib(p_*p^*B\to q_*q^*B)$ where $p\colon\bP_Y(\cE\oplus\cO)\to Y$ and $q\colon\bP_Y(\cE)\to Y$ are the projections. For $A\in\cD(S)$, viewing $A_\MS\in\MS(\cP_S)\otimes\cD(S)$ as a functor $(\MS(\cP_S)^\omega)^\op\to\cD(S)$, Construction \ref{construction-gysin} yields a map
    $$\gys_A(i)\colon (a_Y)_*(A_Y\langle-\N_i\rangle)\to(a_X)_*A_X$$
    for every quasismooth closed immersion $i\colon Y\to X$ in $\cP_S$. In the special case of $X=S$, the above map will be $\gys_A(i)\colon i_*(A_Y\langle-\N_i\rangle)\to A$. In the oriented case, as illustrated in the beginning of \cite[\S 6]{annala-hoyois-iwasa-2023}, if $i$ is of codimension $r$, then $\langle-\N_i\rangle=\langle-r\rangle$ is tensoring with $1_Y\langle-1\rangle^{\otimes r}$. 
\end{construction}

For the rest of this subsection, we take our motivic setup $\cP$ to be all maps. 

\begin{example}\label{represent-classical-cohomologies}
    Construction \ref{represent-cohomology-4ff} applies at least to the following situations:
    \begin{description}
        \item[\'etale] Fix an integer $n$ invertible in $\cO_S$ and a $\ZZ/n$-algebra $\Lambda$. Take $\cD(X)=\Sh_\et(X;\De(\Lambda))$. Then by standard facts on \'etale cohomology, Construction \ref{represent-cohomology-4ff} applies for all $A\in\cD(S)$ with $1_S\langle-1\rangle=1_S(-1)[-2]$ the Tate twist.

        The Kummer map $\B\Gm\to\B^2\mu_n$ of \'etale sheaves gives rise to
        $$c_1(\cL)\in\Gamma(X_\et;\Lambda(1)[2])=\Hom(1_S\langle-1\rangle,(a_X)_*1_X)$$
        functorially for $(X,\cL)\in(\cP_S)_{/\B\Gm}$, and thus a linear orientation of $(1_S)_\MS$. 
        
        Construction \ref{gysin-general-coefficient} now gives the Gysin map $i_*A_Y(-r)[-2r]\to A$ for every quasismooth closed immersion $i\colon Y\to S$ of codimension $r$ and every $A\in\Sh_\et(S;\De(\Lambda))$, recovering Gabber's \'etale Gysin map as in \cite{fujiwara2002absolutepuritygabber}. In fact, Theorem \ref{theorem-oriented-comparision} below will show that our Gysin map generalizes Gabber's Gysin map to arbitrary quasismooth closed immersions of derived schemes. 
        \item[de Rham] Consider the ring stack $\Ga/\Gash(-1)$ over $\bA^1/\Gm$. Via \cite[Remark 2.3.8]{bhatt-f-gauges} this transmutes to a stack $X^{\dRhc,+}$ over $\bA^1/\Gm$ for any stack $X$, geometrizing the absolute Hodge-filtered de Rham cohomology, cf.\ \cite[\S 2.3, \S 2.5]{bhatt-f-gauges}. Take $\cD(X)=\De(X^{\dRhc,+})$. Then by standard facts on de Rham cohomology, Construction \ref{represent-cohomology-4ff} applies for $A=1_S$ with $1_S\langle-1\rangle=1_S(1)[-2]$ the filtrational and cohomological shift. 

        The map $\dlog\colon\Gm[-1]\to\fil^1\dR=[0\to\Omega^1\to\Omega^2\to\cdots]$ of Zariski sheaves gives rise to
        $$c_1(\cL)\in\Gamma(X_\Zar;\fil^1\dR[2])=\Gamma(X^{\dRhc,+};\cO(-1)[2])=\Hom(1_S\langle-1\rangle,(a_X)_*1_X)$$
        functorially for $(X,\cL)\in(\cP_S)_{/\B\Gm}$, and thus a linear orientation of $(1_S)_\MS$. 

        Construction \ref{gysin-general-coefficient} now gives the Gysin map $i_*A_Y(r)[-2r]\to A$ for every quasismooth closed immersion $i\colon Y\to S$ of codimension $r$ and every dualizable $A\in\De(S^{\dRhc,+})$, recovering the Gysin map in \cite[\texttt{0G82}]{stacks}. In fact, Theorem \ref{theorem-oriented-comparision} below will show that our Gysin map coincides with the one constructed loc.\ cit. 
        \item[syntomic] Fix a prime $p$ nilpotent in $\cO_S$. Take $\cD(X)=\De(X^\syn)$ as in \cite[\S 6]{bhatt-f-gauges}. Then by standard facts on syntomic cohomology, Construction \ref{represent-cohomology-4ff} applies for $A=1_S$ with $1_S\langle-1\rangle=1_S\{-1\}[-2]$ the Breuil--Kisin twist as in \cite[Example 6.1.8]{bhatt-f-gauges}. 

        \cite[Proposition 7.5.2]{bhatt-lurie-absoluteprismaticcohomology} gives rise to
        $$c_1(\cL)\in\Gamma_\syn(X;\ZZ_p(1)[2])=\Gamma(X^\syn;\cO\{1\}[2])=\Hom(1_S\langle-1\rangle,(a_X)_*1_X)$$
        functorially for $(X,\cL)\in(\cP_S)_{/\B\Gm}$, and thus a linear orientation of $(1_S)_\MS$. 

        Construction \ref{gysin-general-coefficient} now gives the Gysin map $i_*A_Y\{-r\}[-2r]\to A$ for every quasismooth closed immersion $i\colon Y\to S$ of codimension $r$ and every dualizable $A\in\De(S^\syn)$, recovering the Gysin map in \cite[Definition 5.32]{tang2024syntomic}. In fact, Theorem \ref{theorem-oriented-comparision} below will show that our Gysin map coincides with the one constructed loc.\ cit. 
    \end{description}
    In fact, in the de Rham and syntomic cases, we also expect that Construction \ref{represent-cohomology-4ff} and hence Construction \ref{gysin-general-coefficient} apply for all $A\in\cD(S)$ rather than just dualizable ones. This will be a consequence of suitable base change and projection formula, whose proof might involve corresponding solid stacks as in \cite{kipp2026sixfunctorformalismsyntomiccohomology}. 
\end{example}

\begin{remark}[functoriality in cohomology theories]
    One advantage of our construction is that it is automatically functorial in maps between cohomology theories, because Constructions \ref{represent-cohomology-in-ms} and \ref{represent-cohomology-4ff} are clearly functorial. In other words, maps between cohomology theories automatically respect our Gysin map. 
\end{remark}

\subsection{Comparison for oriented theories}
The following theorem has been invoked several times in the previous subsection; it says that every reasonable construction of Gysin map coincides with ours for oriented theories:

\begin{theorem}[oriented comparison]\label{theorem-oriented-comparision}
    Let $\cC$ be a stable category and $F$ be an exact functor from $(\PS(\cP_S)^\omega)^\op$ to $\cC$. Fix a \emph{weak orientation} of $F$, which here means isomorphisms
    $$t_F(\cE)\colon F(\Sigma_{\bP^1}^rY)=F(\Th_Y(\cE))$$
    for $Y\in\cP_S$ and $\cE$ a locally free sheaf of rank $r$ on $Y$, for all $r\in\NN$, functorial in $(Y,\cE)$ along pullbacks, where $\Sigma_{\bP^1}^rY$ means $Y\wedge(\bP^1/\infty)^{\wedge r}$, which by \cite[\S 3]{annala-hoyois-iwasa-2023} is also $Y\wedge\bP^r/\bP^{r-1}$. Assume that there is a map
    $$g_F(i)\colon F(\Sigma_{\bP^1}^rY)\to F(X)$$
    for every quasismooth closed immersion $i\colon Y\to X$ in $\cP_S$ functorially in pullbacks, and that it coincides functorially with $t_F(\cE)$ for zero sections of vector bundles; in other words, functorially for $(Y,\cE)$, the map $g_F(0_\cE)$ for the zero section $0_\cE\colon Y=\bP_Y(\cO)\to\bP_Y(\cE\oplus\cO)$ is identified with the composite
    $$F(\Sigma_{\bP^1}^rY)\xrightarrow{t_F(\cE)} F(\Th_Y(\cE))\to F(\bP_Y(\cE\oplus\cO))$$
    where the second arrow is induced by the quotient map $\bP_Y(\cE\oplus\cO)\to\Th_Y(\cE)$. 
    
    Then there is a natural isomorphism $g_F(i)=F(\gys(i))\circ t_F(\N_i)$.
\end{theorem}

\begin{proof}
    For a quasismooth closed immersion $i\colon Y\to X$ in $\cP_S$ of codimension $r$, consider the commutative diagram
    $$\begin{tikzcd}
        &&F(\Sigma_{\bP^1}^rY)\ar[r, "t_F(\N_i)"]&F(\Th_Y(\N_i))\ar[dd, "F(\gys(i))"]\\
        F(\Sigma_{\bP^1}^rY)\ar[r]\ar[urr, "\id"]\ar[drr, "\id"']&F(\Sigma_{\bP^1}^r(Y\times\bP^1))\ar[r, "g_F(i_\Q)"]\ar[ur]\ar[dr]&F(\Q(i))\ar[ur]\ar[dr]&\\
        &F(X\wedge\bP^1/0)\ar[ur, crossing over]&F(\Sigma_{\bP^1}^rY)\ar[r,"g_F(i)"']&F(X)
    \end{tikzcd}$$
    where:
    \begin{itemize}
        \item the leftmost map $F(\Sigma_{\bP^1}^rY)\to F(\Sigma_{\bP^1}^r(Y\times\bP^1))$ is induced by the projection $Y\times\bP^1\to Y$, and the upper and lower maps $F(\Sigma_{\bP^1}^r(Y\times\bP^1))\to F(\Sigma_{\bP^1}^rY)$ are induced by the closed immersions $Y=Y\times0\to Y\times\bP^1$ and $Y=Y\times1\to Y\times\bP^1$, respectively; so the two obtuse triangles on the left obviously commute;
        \item the sequence $F(X\wedge\bP^1/0)\to F(\Q(i))\to F(\Th_Y(\N_i))$ is induced by the fundamental cofiber sequence, and the map $F(\Q(i))\to F(X)$ is induced by the closed immersion $X=X\times1\to\Bl_{Y\times0}(X\times\bP^1)$, as in Construction \ref{construction-gysin}; so the triangle on the right commutes by the construction of $\gys(i)$; 
        \item $i_\Q\colon Y\times\bP^1=\Bl_{Y\times0}(Y\times\bP^1)\to\Bl_{Y\times0}(X\times\bP^1)$ denotes the closed immersion of the strict transform of $Y\times\bP^1$, so the two parallelograms commute due to the assumptions imposed (or rather the structures assumed) on $g_F$. 
    \end{itemize}
    The theorem then follows from going through the uppermost and the lowermost routes from the leftmost $F(\Sigma_{\bP^1}^rY)$ to the lower right $F(X)$. 
\end{proof}

\begin{remark}[what is actually used in Theorem \ref{theorem-oriented-comparision}]
    Keep the above notation. If there were no orientation for $F$ but only a map
    $$g'_F(i)\colon F(\Th_Y(\N_i))\to F(X)$$
    functorially for every quasismooth closed immersion $i\colon Y\to X$ in $\cP_S$, then we would still be able to draw the commutative diagram
    $$\begin{tikzcd}
        &F(\Th_Y(\N_i))\ar[r, "g'_F(0_{\N_i})"]&F(\Th_Y(\N_i))\ar[dd, "F(\gys(i))"]\\
        F(\Th_{Y\times\bP^1}(\N_i\boxtimes\cO(1)))\ar[r, "g'_F(i_\Q)"]\ar[ur]\ar[dr]&F(\Q(i))\ar[ur]\ar[dr]&\\
        F(X\wedge\bP^1/0)\ar[ur, crossing over]&F(\Th_Y(\N_i))\ar[r,"g'_F(i)"']&F(X)
    \end{tikzcd}$$
    where it is not hard to see that $\N_{\Bl_{Y\times0}(Y\times\bP^1)/\Bl_{Y\times0}(X\times\bP^1)}=\N_i\boxtimes\cO(1)$, and the upper and lower maps $F(\Th_{Y\times\bP^1}(\N_i\boxtimes\cO(1)))\to F(\Th_Y(\N_i))$ are still induced by the closed immersions $Y=Y\times0\to Y\times\bP^1$ and $Y=Y\times1\to Y\times\bP^1$, respectively. However, at least a priori, there is no projection map $\Th_{Y\times\bP^1}(\N_i\boxtimes\cO(1))\to\Th_Y(\N_i)$ in $\PS(\cP_S)$ for the proof of Theorem \ref{theorem-oriented-comparision} to work, and this is where the orientation comes in. In other words, the orientation is used only to produce a common section $F(\Th_Y(\N_i))\to F(\Th_{Y\times\bP^1}(\N_i\boxtimes\cO(1)))$ to the above two maps. 
\end{remark}

\begin{remark}\label{remark-orientation-compare-thom}
    As illustrated in the beginning of \cite[\S 6]{annala-hoyois-iwasa-2023}, an orientation in the sense of \cite[Definition 3.1.2]{annala-iwasa-2022} gives rise to a weak orientation in the sense of Theorem \ref{theorem-oriented-comparision}. As far as I know, weak orientations always occur this way in practice. 
\end{remark}

\appendix

\section{Blowup in derived algebraic geometry}\label{appendix-dag}

In this appendix, we collect definitions and properties of nested blowups of some types of diagrams in derived algebraic geometry used in the main body of the paper, partially generalizing \cite{hekkingkhanrydh2025deformationnormalbundleblowups}. We start with a simple but useful fact. 

\begin{proposition}[subtracting a divisor]\label{equivalence-subtract-divisor}
    Let $X$ be a stack and $D\to X$ be a divisor. Then there is an adjunction
    $$\{\text{closed immersion }Z\to X\}\rightleftarrows\{D\to Y\to X\mid Y\to X\text{ is a closed immersion}\}$$
    where the left adjoint is $Z\mapsto Z+D$ (see \S \ref{dag-notation}(\ref{dag-notation:adding-closed-immersion})) and the right adjoint is $Y\mapsto\Bl_DY$. In this case, we denote $\Bl_DY$ by $Y-D$. Moreover, this adjunction is an equivalence between divisors on both sides.
\end{proposition}

\begin{proof}
    Clearly, it suffices to treat the case where $X=\Spec(R)$ and $D=\Spec(R/r)$. Then a $D\to Y\to X$ where $Y\to X$ is a closed immersion corresponds to a factorization $R\to A\to R/r$ where $R\to A$ is surjective. 
    
    We first show that the claimed functor is well-defined, namely that $\Bl_DY\to X$ is indeed a closed immersion. Let $I=\fib(R\to A)$. Recall from \cite[Definition 4.3]{hekkingkhanrydh2025deformationnormalbundleblowups} that $\Bl_DY$, as a functor, takes a ring $S$ to the anima of excessive diagrams
    $$\begin{tikzcd}
        A\ar[r]\ar[d]&R/r\ar[d]\\
        S\ar[r]&S/J
    \end{tikzcd}$$
    where $J$ is an invertible ideal of $S$. Recall that the diagram being excessive means $\fib(A\to R/r)\otimes_AS\twoheadrightarrow J$, or equivalently $rR\otimes_RS\twoheadrightarrow J$, or equivalently $rS=J$. Therefore, we can rewrite the above anima as that of diagrams
    $$\begin{tikzcd}
        R\ar[r]\ar[d]&A\ar[r]\ar[ld]&R/r\ar[d]\\
        S\ar[rr]&&S/rS
    \end{tikzcd}$$
    Now this manifestly takes colimits of $A$ as $R$-algebras to limits of $R$-stacks, so in order to prove that $\Bl_DY\to X$ is a closed immersion, it suffices to treat the case where $A=R/rg$ for some $g\in R$ and $A\to R/r$ is the obvious map, as these generate all surjections from $R$ under colimits. Now this is \cite[Lemma 3.4]{tang2024slicingcriterionindsmoothring}, which says that $\Bl_DY=\Spec(R/g)$. 

    From the above discussion, it is clear that $Y\mapsto\Bl_DY$ is right adjoint to $Z\mapsto Z+D$. It is also easy to see that $Z\mapsto Z+D$ is conservative. So it remains to see $Y=\Bl_DY+D$ when $Y$ is a divisor, and we can again assume that $Y=\Spec(R/rg)$, and it becomes the obvious fact that $R/rg=R/g\times_{R/(g,r)}R/r$.
\end{proof}

\begin{remark}
    In classical algebraic geometry, i.e.\ if $X$ is a classical scheme and $D\to X$ is a classical divisor locally cut out by a nonzerodivisor, then the adjunction in Proposition \ref{equivalence-subtract-divisor} is actually an equivalence. However, this is not true in derived algebraic geometry. In fact, $Z\mapsto Z+D$ does not preserve products as seen in the following simplest example:
    \begin{itemize}
        \item Let $X=\Spec(R)$, $D=\Spec(R/r)$, $Z_1=\Spec(R/g_1)$ and $Z_2=\Spec(R/g_2)$. Then $Z_1\times_XZ_2+D$ is the spectrum of the ring
        $$R/(g_1,g_2)r:=R/(g_1,g_2)\times_{R/(g_1,g_2,r)}R/r$$
        instead of $R/(g_1r,g_2r)$ whose spectrum is $(Z_1+D)\times_X(Z_2+D)$. The two rings $R/(g_1,g_2)r$ and $R/(g_1r,g_2r)$ have the same $\pi_0$, but are in general different. For example when $g_1=g_2=1$ and $r=0$, they are $R/0=R\oplus R[1]$ and $R/(0,0)=R/0\otimes_RR/0=R\oplus R^{\oplus2}[1]\oplus R[2]$, respectively. 
    \end{itemize}
\end{remark}

\begin{corollary}\label{blowup-subtract-divisor}
    Let $D\to Y\to X$ be closed immersions of stacks where $D\to X$ is a divisor. Then $\Bl_YX=\Bl_{Y-D}X$. More precisely, 
    $$\begin{tikzcd}
        \E_{Y-D}X+D|_{\Bl_{Y-D}X}\ar[r]\ar[d]&\Bl_{Y-D}X\ar[d]\\
        Y\ar[r]&X
    \end{tikzcd}$$
    is the universal diagram defining $\Bl_YX$. 
\end{corollary}

\begin{proof}
    We use the normal deformation studied in \cite[\S 3]{hekkingkhanrydh2025deformationnormalbundleblowups}. Let
    $$\begin{tikzcd}
        \fN_YX\ar[r]\ar[d]&\fD_YX\ar[d]\\
        Y\ar[r]&X
    \end{tikzcd}$$
    be the defining diagram of the normal deformation of $Y\to X$, and similar for $Y-D\to X$. Then the defining universal property of the normal deformation gives a map $D|_{\fD_YX}\to\fN_YX$ of divisors in $\fD_YX$, and using Proposition \ref{equivalence-subtract-divisor} one easily sees an isomorphism $\fD_YX=\fD_{Y-D}X$ which takes $\fN_YX$ to $\fN_{Y-D}X+D|_{\fD_{Y-D}X}$. Now the corollary is reduced to the following easy observation: 
    \begin{itemize}
        \item Let $X'$ be a stack over $X$, $D'\to X'$ be the base change of $D\to X$, and $Z'\to X'$ be a closed immersion. Then a map
        $$\begin{tikzcd}
            Z'+D'\ar[r]\ar[d]&X'\ar[d]\\
            Y\ar[r]&X
        \end{tikzcd}$$
        is excessive if and only if
        $$\begin{tikzcd}
            Z'\ar[r]\ar[d]&X'\ar[d]\\
            Y-D\ar[r]&X
        \end{tikzcd}$$
        is excessive. \qedhere
    \end{itemize}
\end{proof}

\begin{corollary}\label{total-transform-excessive}
    Let $Z\to Y\to X$ be closed immersions of stacks. Then
    $$\Bl_ZY=\Bl_{\E_ZX}(Y|_{\Bl_ZX})=Y|_{\Bl_ZX}-\E_ZX\text{ and }\E_ZY=\E_{\E_ZX}(Y|_{\Bl_ZX}).$$
\end{corollary}

\begin{proof}
    Note that $\Bl_ZY$ naturally lives over $\Bl_ZX$, because the square
    $$\begin{tikzcd}
        Z\ar[r]\ar[d]&Y\ar[d]\\
        Z\ar[r]&X
    \end{tikzcd}$$
    is obviously excessive. The corollary then follows from Proposition \ref{equivalence-subtract-divisor} by viewing both sides as $\Bl_ZX$-stacks and comparing their defining universal properties.
\end{proof}

\begin{definition}[strict transform]
    In particular, in the setting of Corollary \ref{total-transform-excessive}, $\Bl_ZY\to\Bl_ZX$ is a closed immersion, which we call the \emph{strict transform} of $Y$ in $\Bl_ZX$. Note that the strict transform depends on the map $Z\to Y$ of stacks over $X$, which is really an extra datum, unlike in classical algebraic geometry. 
\end{definition}

\subsection{Blowing up a poset of closed immersions}
The following definition is central in the symmetric monoidality of the Gysin map. 

\begin{definition}[blowing up a poset of closed immersions]\label{definition-blowup-poset}
    Let $P$ be a finite poset. Let $X$ be a stack and $i=(i_p\colon Y_p\to X)_{p\in P}$ a $P$-diagram of closed immersions to $X$. Generalizing Definition \ref{definition-excessive-map-arrow}, for a stack $X'$ and a $P$-diagram $i'=(i'_p\colon Y'_p\to X')_{p\in P}$, a map $f\colon i'\to i$ is said to be \emph{excessive} if for every $p\in P$,
    $$\begin{tikzcd}
        Y'_p\ar[r]\ar[d]&X'\ar[d,"f"]\\
        Y_p\ar[r]&X
    \end{tikzcd}$$
    is excessive, i.e.\ $f^*\cI_{Y_p/X}\twoheadrightarrow\cI_{Y'_p/X'}$. The \emph{blowup} of $i$, denoted $\Bl(i)$ or $\Bl_YX$, is defined as the final stack along with a $P$-diagram of divisors, denoted $(\E_p(i))_{p\in P}$ or $(\E_pX)_{p\in P}$, and an excessive map of $P$-diagrams from $(\E_p(i)\to\Bl(i))_{p\in P}$ to $i$. We call $(\E_p(i))_{p\in P}$ the \emph{exceptional divisors}, and more specifically call $\E_p(i)$ the \emph{exceptional divisor over $Y_p$}. 
\end{definition}

\begin{remark}\label{remark-compare-mayeux}
    \cite[Definition 3.26]{mayeux2025multigradedprojschemes} defines a similar notion in classical algebraic geometry, where only the set of closed subschemes matters. Here in derived algebraic geometry, the poset structure really matters. For example, take a closed immersion $Y\to X$, let $P=\{0,1\}$, and let $Y_0=Y_1=Y$; if we view $P$ as a discrete poset, the resulting blowup is $\Bl_YX\times_X\Bl_YX$; if we instead view $P$ as the usual poset where $0\le 1$, the resulting blowup is $\Bl_YX$; the two are in general different. 
\end{remark}

    

\begin{remark}[functoriality of blowup]\label{functoriality-blowup-poset}
    Definition \ref{definition-blowup-poset} has the following functorialities (keeping the notation there):
    \begin{enumerate}
        \item Excessive maps of $P$-diagrams of closed immersions induce maps of blowups. 
        \item\label{functoriality-blowup-poset:base-change} For $X'\to X$, let $i'$ denote the base change of $i$ to $X'$; then $\Bl(i')=\Bl(i)\times_XX'$, and $\E_p(i')=\E_p(i)\times_XX'$ for all $p\in P$. 
        \item\label{functoriality-blowup-poset:change-poset} For $Q\subseteq P$, there is a canonical map $\Bl(i)\to\Bl(i|_Q)$. 
    \end{enumerate}
\end{remark}

\begin{lemma}\label{blowup-poset-subtract-divisor}
    Let $P$ be a finite poset and $i=(i_p\colon Y_p\to X)_{p\in P}$ be a $P$-diagram of closed immersions. Suppose $0\in P$ is an element with $Y_0=D$ a divisor in $X$. Define $i'=(i'_p\colon Y'_p\to X)_{p\in P\setminus0}$ as
    $$Y'_p=\begin{cases}
        Y_p-D,&p>0;\\
        Y_p,&\text{otherwise}.
    \end{cases}$$
    Then $\Bl(i)=\Bl(i')$, where
    $$\E_p(i)=\begin{cases}
        D|_{\Bl(i')},&p=0;\\
        \E_p(i')+D|_{\Bl(i')},&p>0;\\
        \E_p(i'),&\text{otherwise}.
    \end{cases}$$
\end{lemma}

\begin{proof}
    This is an immediate consequence of Definition \ref{definition-blowup-poset} and Corollary \ref{blowup-subtract-divisor}. 
\end{proof}

\subsection{Representability of the poset blowup}
In this subsection, we show that in practice the blowup in Definition \ref{definition-blowup-poset} is reasonably representable. The main results are Theorem \ref{theorem-excessive-diagram-blowup-representable} and Theorem \ref{theorem-blowup-cofiber-comparison} at the end of this subsection, which say that for a motivic setup $\cP$ as in Definition \ref{definition-motivic-setup} and a base scheme $S$, the blowup of an excessive diagram in $\cP_S$ stays in $\cP_S$, and blowing up an excessive diagram does not change the total cofiber of the diagram in $\Sh_{\Zar,\qbu}(\cP_S)$. 

Recall that a \emph{lattice} means a poset that has finite infimums and supremums. For any lattice, we often denote its minimum (empty supremum) by $0$ and maximum (empty infimum) by $1$. 

\begin{definition}[excessive diagram]\label{definition-excessive-diagram}
    Let $L$ be a finite lattice and $Y=(Y_\ell)_{\ell\in L}$ be an $L$-diagram of stacks. We say that $Y$ is \emph{excessive} if: 
    \begin{enumerate}
        \item all maps in the diagram are affine;
        \item for any downward-closed subset $K\subseteq L$, the derived $\cO_{Y_1}$-algebra
        $$\lim_{k\in K^\op}\cO_{Y_k}$$
        is connective, and the structure map $\cO_{Y_1}\to\lim_{k\in K}\cO_{Y_k}$ is surjective. 
    \end{enumerate}
    In this case we let $Y_K$ denote the spectrum over $Y_1$ of the above limit, which is closed immersed to $Y_1$. 
    
    If $X$ is a stack, we denote the category of excessive $L$-diagram of stacks over $X$ by $\Exc_L\Stack_X$. If $R$ is a ring, we denote the category of excessive $L$-diagram of $R$-algebras by $\Exc_L\Ring_R$, which by definition is the full subcategory of $\Exc_L\Stack_R^\op$ consisting of the affines. 
\end{definition}

There are many equivalent characterizations of excessivity, coming from the following general fact for diagrams in stable categories with t-structures. 

\begin{proposition}\label{excessive-diagram-characterization}
    Let $L$ be a finite lattice and $\cC$ be a stable category with a t-structure $(\cC_{\ge0},\cC_{\le0})$. Let $C=(C_\ell)_{\ell\in L}$ be an $L^\op$-diagram in $\cC$. The following are equivalent:
    \begin{enumerate}
        \item\label{excessive-diagram-characterization:definition} For all downward-closed $K\subseteq L$, $\fib(C_1\to\lim_{k\in K}C_k)$ is connective. 
        \item\label{excessive-diagram-characterization:generation} For all $\ell\in L$, $\fib(C_1\to\lim_{k\not\ge\ell}C_k)$ is connective
        \item\label{excessive-diagram-characterization:point} For all $\ell\in L$, $\fib(C_\ell\to\lim_{k<\ell}C_k)$ is connective. 
        \item\label{excessive-diagram-characterization:two-downward-sets} For all downward-closed $J\subseteq K\subseteq L$, 
        $$\fib\left(\lim_{k\in K}C_k\to\lim_{j\in J}C_j\right)$$
        is connective. 
    \end{enumerate}
    In particular, for $Y=(Y_\ell)$ an $L$-diagram of stacks with affine transition maps, taking $\cC=\De(Y_1)$ and $C_\ell=\cO_{Y_\ell}$, we get corresponding characterizations of excessivity. 
\end{proposition}

\begin{proof}
    Equip $L$ with the topology where downward-closed means open, so that it becomes a finite spectral space. View $C$ as a sheaf on $L$ where $C_\ell=C(L_{\le\ell})$. Then for all downward-closed $J\subseteq K\subseteq L$, we have $\lim_{k\in K}C_k=C(K)$, and
    $$\fib\left(\lim_{k\in K}C_k\to\lim_{j\in J}C_j\right)=\Gamma_{K\setminus J}(C)$$
    is the local cohomology on the locally closed subset $K\setminus J$. Recall the following property of local cohomology: for locally closed $U\subseteq V\subseteq L$ where $U$ is open in $V$, there is naturally a fiber sequence
    $$\Gamma_{V\setminus U}(C)\to\Gamma_V(C)\to\Gamma_U(C).$$
    
    With these observations in hand, we now start the proof. First, note that $(\ref{excessive-diagram-characterization:two-downward-sets})\implies(\ref{excessive-diagram-characterization:definition})\implies(\ref{excessive-diagram-characterization:generation})$ and $(\ref{excessive-diagram-characterization:two-downward-sets})\implies(\ref{excessive-diagram-characterization:point})$ are trivial. 
    \begin{description}
        \item[$(\ref{excessive-diagram-characterization:point})\implies(\ref{excessive-diagram-characterization:two-downward-sets})$] Let $V=K\setminus J$. We have noticed that the fiber in (\ref{excessive-diagram-characterization:two-downward-sets}) is $\Gamma_V(C)$. We do induction on the size of $V$. The case $V=\varnothing$ is obvious. Otherwise, let $\ell$ be a minimal element in $V$, so $U=\{\ell\}$ is open in $V$. Then (\ref{excessive-diagram-characterization:point}) says that $\Gamma_U(C)$ is connective, so (\ref{excessive-diagram-characterization:two-downward-sets}) follows from induction hypothesis and the fiber sequence displayed above. 
        \item[$(\ref{excessive-diagram-characterization:generation})\implies(\ref{excessive-diagram-characterization:point})$] We do induction on the size of an upward-closed $Z\subseteq L$ to show:
        \begin{itemize}
            \item For all upward-closed $W\subseteq Z$, $\Gamma_W(C)$ is connective, and for all $\ell\in Z$, $\Gamma_{\{\ell\}}(C)$ is connective. 
        \end{itemize}
        The case $Z=\varnothing$ is obvious. Otherwise, let $\ell$ be a minimal element in $Z$. By the induction hypothesis on $Z\setminus\ell$, it suffices to see that $\Gamma_{\{\ell\}}(C)$ is connective. This follows from (\ref{excessive-diagram-characterization:generation}) and the induction hypothesis, using the fiber sequence displayed above applied on $\{\ell\}\subseteq L_{\ge\ell}\subseteq Z$, because by upward-closedness of $Z$ we have $L_{\ge\ell}\subseteq Z$ and $L_{\ge\ell}\setminus\ell\subseteq Z\setminus\ell$. \qedhere
    \end{description}
\end{proof}

\begin{corollary}\label{excessive-diagram-restrict}
    Let $K$ and $L$ be finite lattices, and $Y=(Y_\ell)_{\ell\in L}$ be an excessive $L$-diagram. Let $f\colon K\to L$ be a map of posets that preserves nonempty infimums. Then $Y|_K$ is excessive. 
\end{corollary}

\begin{proof}
    It suffices to show that any downward-closed subset $I$ of $K$ is cofinal in the downward closure $J$ of $f(I)$ in $L$, i.e.\ that for any $j\in J$, the anima $|I\cap f^{-1}(J_{\ge j})|$ is contractible. This is clear, as $I\cap f^{-1}(J_{\ge j})$ is nonempty and is closed under nonempty infimums. 
\end{proof}

\begin{corollary}\label{excessive-diagram-product}
    Let $\Omega$ be a finite set and $X$ be a stack. For each $\omega\in\Omega$, let $L^\omega$ be a finite lattice, and let $Y^\omega=(Y^\omega_{\ell^\omega})_{\ell^\omega\in L^\omega}$ be an excessive $L^\omega$-diagram with $Y^\omega_1=X$. Let $L^\Omega=\prod_{\omega\in\Omega}L^\omega$ and let $Y^\Omega$ be the $L^\Omega$-diagram where $Y^\Omega_\ell=\prod_{\omega\in\Omega}Y^\omega_{\ell^\omega}$ for $\ell=(\ell^\omega)_{\omega\in\Omega}\in L^\Omega$, where the products are taken over $X$. Then $Y^\Omega$ is excessive. 
\end{corollary}

\begin{proof}
    This is immediate from Proposition \ref{excessive-diagram-characterization} because for $\ell=(\ell^\omega)_{\omega\in\Omega}\in L^\Omega$, 
    $$\fib\left(\cO_{Y_\ell}\to\lim_{k<\ell}\cO_{Y_k}\right)=\bigotimes_{\omega\in\Omega}\fib\left(\cO_{Y^\omega_{\ell^\omega}}\to\lim_{k^\omega<\ell^\omega}\cO_{Y^\omega_{k^\omega}}\right),$$
    where the tensor product is over $\cO_X$. 
\end{proof}

\begin{corollary}\label{excess-diagram-limit-pointwise}
    Let $X$ be a stack and $L$ be a finite lattice. Then the category $\Exc_L\Stack_X$ has all limits which can be computed pointwise. 
\end{corollary}

\begin{proof}
    By base change, one immediately reduces the claim to the same claim on the subcategory of those $Y$ with $Y_1=X$. It suffices to treat finite products and sifted limits. The case of finite products follows from Corollaries \ref{excessive-diagram-product} and \ref{excessive-diagram-restrict}. Pointwise sifted limits in question are computed by pointwise sifted colimits of $\cO_X$-algebras, which commute with forgetting to $\cO_X$-modules. Since excessivity is a connectivity condition, it is clearly preserved by $\cO_X$-module colimits. 
\end{proof}

\begin{proposition}\label{excess-diagram-generation}
    Let $L$ be a finite lattice and let $R$ be a ring. For $k\in L$, define the $L^\op$-diagram $R[x_k]$ of $R$-algebras as
    $$R[x_k]_\ell=\begin{cases}
        R[x],&\ell\ge k,\\
        R,&\text{otherwise},
    \end{cases}$$
    whose transition map is either the identity or $x\mapsto0$. Then $(R[x_k])_{k\in L}$ is a family of compact projective generators of $\Exc_L\Ring_R$. 
\end{proposition}

\begin{proof}
    For a fixed $k\in L$, note that: 
    \begin{itemize}
        \item $R[x_k]$ is excessive. To see this, take any downward-closed $J\subseteq L$ and we have to see that $R[x]\twoheadrightarrow\lim_{j\in J}R[x_k]_j$. If $k\notin J$ it is obvious. If $k\in J$, let $I=J\setminus J_{\ge k}$, $J'$ be the downward closure of $J_{\ge k}$, and $I'=I\cap J'$. Then
        $$\lim_{j\in J}R[x_k]_j=\lim_{i\in I}R[x_k]_i\times_{\lim_{i\in I'}R[x_k]_i}\lim_{j\in J'}R[x_k]_j=R\times_RR[x]=R[x],$$
        as $J_{\ge k}$ is obviously cofinal in $J'$. This gives the desired surjectivity. 
        \item For $A\in\Exc_L\Ring_R$, 
        $$\Hom(R[x_k],A)=\fib\left(A_1\to\lim_{\ell\not\ge k}A_\ell\right)$$
        which is also the fiber in $\De(R)$ by excessivity, and this commutes with sifted colimits of $A$, so $R[x_k]$ is compact projective. 
    \end{itemize}
    It remains to prove generation. By Corollary \ref{excess-diagram-limit-pointwise} it is easy to see that $\Exc_L\Ring_R$ is presentable, so generation is equivalent to separation, which amounts to saying that if a map $f\colon A\to B$ in $\Exc_L\Ring_R$ induces equivalences for the above fibers for all $k\in L$, then $f$ is an equivalence. This is Lemma \ref{lemma-fiber-equivalence-implies-equivalence} below, applied on $\fib(f)$ taken as diagrams in $\De(R)$. 
\end{proof}

\begin{lemma}\label{lemma-fiber-equivalence-implies-equivalence}
    Let $L$ be a finite lattice, $\cC$ a stable category, and $C=(C_\ell)_{\ell\in L}$ an $L^\op$-diagram in $\cC$. If for all $k\in L$, 
    $$\fib\left(C_1\to\lim_{\ell\not\ge k}C_\ell\right)=0,$$
    then $C=0$. 
\end{lemma}

\begin{proof}
    This follows from Proposition \ref{excessive-diagram-characterization} by taking ``connective'' as ``zero''. 
\end{proof}

\begin{corollary}\label{subtract-divisor-preserve-excessive}
    Let $L$ be a finite lattice and $Y=(Y_\ell)_{\ell\in L}$ be an excessive $L$-diagram of stacks. Let $D\to Y_0$ be a map such that the composite $D\to Y_1$ is a divisor of $Y_1$. Then $Y-D=(Y_\ell-D)_{\ell\in L}$ is also excessive, where we are using the notation of Proposition \ref{equivalence-subtract-divisor}, viewing everything as closed immersed into $Y_1$. 
\end{corollary}

\begin{proof}
    Working locally on $Y_1$, we can assume $Y_1=\Spec(R)$ and $D=\Spec(R/r)$. By Corollary \ref{excess-diagram-limit-pointwise}, Proposition \ref{excess-diagram-generation}, and the fact that ``${}-D$'' is a right adjoint and preserves limits, it is easy to reduce the claim to diagrams of the form
    $$Y_\ell=\Spec\begin{cases}
        R,&\ell\ge k,\\
        R/ar&\text{otherwise},
    \end{cases}$$
    for some $k\in L$. Now the claim is obvious, as $\Spec(R/ar)-D=\Spec(R/a)$. 
\end{proof}

\begin{proposition}\label{excessive-implies-surjective}
    Let $L$ be a finite lattice and $Y=(Y_\ell)_{\ell\in L}$ be an excessive $L$-diagram of stacks. Denote $X=Y_1$. Let $(Q_m)_{m\in M}$ be a finite family of downward-closed subsets of $L$, and $Q=\bigcap_{m\in M}Q_m$. Then the map $\bigoplus_{m\in M}\cI_{Y_{Q_m}/X}\to\cI_{Y_Q/X}$ is surjective. 
\end{proposition}

\begin{proof}
    It suffices to treat the cases $M=\varnothing$ and $M=\{1,2\}$. The former is obvious. For the latter, note that $R=Q_1\cup Q_2$ is also downward-closed, and $\cO_{Y_R}=\cO_{Y_{Q_1}}\times_{\cO_{Y_Q}}\cO_{Y_{Q_2}}$, so $\cI_{Y_R/X}=\cI_{Y_{Q_1}/X}\times_{\cI_{Y_Q/X}}\cI_{Y_{Q_2}/X}$ is connective. This implies the desired surjectivity. 
\end{proof}

\begin{example}\label{example-product-excessive}
    Let $\Omega$ be a finite set and $(Y_\omega\to Y)_{\omega\in\Omega}$ an $\Omega$-family of closed immersions. For $\Psi\subseteq\Omega$, let $Y_\Psi$ be the fiber product of $(Y_\omega)_{\omega\in\Psi}$ over $Y$. We equip the power set $\ps(\Omega)$ with the inverse inclusion partial order, denoted $\ps(\Omega)^\op$. Then the $\ps(\Omega)^\op$-diagram $(Y_\Psi)_{\Psi\subseteq\Omega}$ is excessive by Corollary \ref{excessive-diagram-product}. Moreover, let $Y\to X$ be a closed immersion and let $\ps(\Omega)^\op_+=\ps(\Omega)^\op\sqcup\{1\}$, the lattice obtained from $\ps(\Omega)^\op$ by adding an element and setting it as the maximum. Then if we let $Y_1=X$, it is easy to see that the resulting $\ps(\Omega)^\op_+$-diagram is also excessive. 
\end{example}

Fix a motivic setup $\cP$ as in Definition \ref{definition-motivic-setup} and a base scheme $S$. 

\begin{theorem}\label{theorem-excessive-diagram-blowup-representable}
    Let $L$ be a finite lattice and $Y=(Y_\ell)_{\ell\in L}$ be an excessive $L$-diagram in $\cP_S$ where all the arrows are quasismooth. Denote $\Pi=L\setminus1$ and $X=Y_1$, and view $Y$ as a diagram $i=(i_\pi\colon Y_\pi\to X)_{\pi\in\Pi}$ of closed immersions. Then there is a canonical $\Pi$-family $(\E'_\pi(i))_{\pi\in\Pi}$ of divisors of $\Bl(i)$ such that:
    \begin{enumerate}
        \item For all $\pi\in\Pi$, $\E_\pi(i)=\sum_{\rho\le\pi}\E'_\rho(i)$, and for $\pi\le\pi'$, the map $\E_\pi(i)\to\E_{\pi'}(i)$ is the obvious one coming from the implication $\rho\le\pi\implies\rho\le\pi'$. 
        \item For $\Gamma\subseteq\Pi$ running through all chains in $\Pi$, i.e.\ all totally ordered subsets of $\Pi$, the opens
        $$U_\Gamma=\Bl(i)\setminus\bigcup_{\pi\notin\Gamma}\E'_\pi(i)$$
        cover $\Bl(i)$. In other words, for $\pi,\rho\in\Pi$ incomparable, $\E'_\pi(i)\cap\E'_\rho(i)=\varnothing$. 
        \item\label{theorem-excessive-diagram-blowup-representable:divisors-representable} For all $\Psi\subseteq\Pi$, the fiber product $\E'_\Psi(i)$ of $(\E'_\pi(i))_{\pi\in\Psi}$ over $\Bl(i)$ lies in $\cP_S$. In particular, $\Bl(i)=\E'_\varnothing(i)\in\cP_S$. 
    \end{enumerate}
\end{theorem}

Under the notation above, we call $\E'_\pi(i)$ the \emph{strict exceptional divisor} over $Y_\pi$. 

\begin{proof}
    We do induction on the size of $L$. When $L=\{1\}$ there is nothing to prove. Otherwise, $0\ne1$ and $0\in\Pi$. By Definition \ref{definition-blowup-poset}, $\Bl(i)\to X$ factors through $\Bl_{Y_0}X$, and in fact $\Bl(i)=\Bl(i|_{\Bl_{Y_0}X})$. By Corollaries \ref{blowup-subtract-divisor} and \ref{total-transform-excessive} applied on the divisor $\E_{Y_0}X$, it is also the blowup of $(\Bl_{Y_0}Y_\pi\to\Bl_{Y_0}X)_{\pi\in\Pi}$, in which $\E_0(i)$ is the blowup of $(\E_{Y_0}Y_\pi\to\E_{Y_0}X)_{\pi\in\Pi}$. These diagrams stay in $\cP_S$ by Definition \ref{definition-motivic-setup} and stay excessive by Corollary \ref{subtract-divisor-preserve-excessive}, so we are reduced to the case $Y_0=\varnothing$, where I claim that Zariski locally the theorem reduces to the induction hypothesis: 
    
    For $\pi\in L\setminus0$, let $U_\pi\subseteq X$ be the open complement of the union of $Y_\rho$ for all $\rho\not\ge\pi$. Then over $U_\pi$, the induction hypothesis applies with the smaller lattice $L_{\ge\pi}$. Now it remains to see that $(U_\pi)_{\pi\in\Pi\setminus0}$ cover $X$, or equivalently:
    \begin{itemize}
        \item For $M\subseteq\Pi\setminus0$ with $\inf(M)=0$, the intersection of $(Y_m)_{m\in M}$ is empty. 
    \end{itemize}
    This follows immediately from Proposition \ref{excessive-implies-surjective} where one takes $Q_m=\Pi_{\le m}$. 
\end{proof}

\begin{remark}
    Note that Definition \ref{definition-excessive-diagram} does not need the indexing poset to be a lattice. Theorem \ref{theorem-excessive-diagram-blowup-representable} might also hold in the same generality, but at this point the author is not sure about this. 
\end{remark}

\begin{theorem}\label{theorem-blowup-cofiber-comparison}
    Keep the notation in Theorem \ref{theorem-excessive-diagram-blowup-representable}. Consider the $\ps(\Pi)^\op$-diagram $\E'=(\E'_\Psi(i))_{\Psi\subseteq\Pi}$ resulting from Theorem \ref{theorem-excessive-diagram-blowup-representable}(\ref{theorem-excessive-diagram-blowup-representable:divisors-representable}). Let $\inf\colon\ps(\Pi)^\op\to L$ be the map of posets as the notation suggests. Then there is a natural map $\E'\to\inf^*Y$ of $\ps(\Pi)^\op$-diagrams in $\cP_S$, inducing an isomorphism on the total cofibers in $\Sh_{\Zar,\qbu}(\cP_S)$. Equivalently, one can form the left Kan extension $\inf_!\E'$ as an $L$-diagram in $\Sh_\Zar(\cP_S)$. Then in fact $(\inf_!\E')_\pi=\E_\pi$, which we view as
    $$\sum_{\rho\le\pi}\E'_\rho:=\colim_{\varnothing\ne\Psi\subseteq\Pi_{\le\pi}}\E'_\Psi\in\Sh_\Zar(\cP_S),$$
    and the natural map $\inf_!\E'\to Y$ induces an isomorphism on the total cofibers in $\Sh_{\Zar,\qbu}(\cP_S)$. In particular, the total cofiber of $\E'$ and that of $Y$ in $\Sh_{\Zar,\qbu}(\cP_S)$ are canonically isomorphic. 
\end{theorem}

\begin{proof}
    We use the same induction process as in Theorem \ref{theorem-excessive-diagram-blowup-representable}. When $L=\{1\}$ there is still nothing to prove. Otherwise, $0\ne1$ and $0\in\Pi$. By Theorem \ref{theorem-excessive-diagram-blowup-representable} or rather its proof, the map $\E'\to\inf^*Y$ of $\ps(\Pi)^\op$-diagrams factors through the $\ps(\Pi)^\op$-diagram $F=(F_\Psi)_{\Psi\subseteq\Pi}$ where
    $$F_\Psi=\begin{cases}
        \E_{Y_0}Y_{\inf(\Psi\setminus0)},&0\in\Psi;\\
        \Bl_{Y_0}Y_{\inf(\Psi)},&0\notin\Psi.
    \end{cases}$$
    As in the proof of Theorem \ref{theorem-excessive-diagram-blowup-representable}, Zariski locally on $F_\varnothing=\Bl_{Y_0}X$, the induction hypothesis shows that the map $\E'\to F$ of $\ps(\Pi)^\op$-diagrams induces an isomorphism on the total cofibers. Therefore, it remains to show that $F\to\inf^*Y$ does too. For this, it suffices to show that for each $\Psi\subseteq\Pi\setminus0$, the square
    $$\begin{tikzcd}
        \E_{Y_0}Y_{\inf(\Psi)}\ar[r]\ar[d]&\Bl_{Y_0}Y_{\inf(\Psi)}\ar[d]\\
        Y_{\inf(\Psi\cup\{0\})}\ar[r]&Y_{\inf(\Psi)}
    \end{tikzcd}$$
    is a pushout in $\Sh_{\Zar,\qbu}(\cP_S)$. This follows immediately from the definition as $Y_{\inf(\Psi\cup\{0\})}=Y_0$. 
\end{proof}

\begin{corollary}\label{corollary-blowup-cofiber-comparison-minimum}
    Keep the notation in the proof of Theorem \ref{theorem-blowup-cofiber-comparison}, and consider the restrictions of the diagrams $\E'$ and $F$ from $\ps(\Pi)^\op$ to $\{\Psi\subseteq\Pi\mid0\in\Psi\}^\op$. Then the map $\E'\to F$ induces an isomorphism on the total cofibers of the restrictions. In particular, the total cofiber of $(\E'_\Psi(i))_{0\in\Psi\subseteq\Pi}$ and that of $(\E_{Y_0}Y_\ell)_{\ell\in L}$ are canonically isomorphic. 
\end{corollary}

\begin{proof}
    By the proof of Theorem \ref{theorem-excessive-diagram-blowup-representable}: 
    \begin{itemize}
        \item Passing to an open cover of $\E_{Y_0}X$, one can assume that $L\setminus0$ has a minimum, and thus is still a lattice. 
        \item The restricted diagram $(\E'_\Psi(i))_{0\in\Psi\subseteq\Pi}$ is obtained by applying Theorem \ref{theorem-excessive-diagram-blowup-representable}(\ref{theorem-excessive-diagram-blowup-representable:divisors-representable}) on the $(L\setminus0)$-diagram $(\E_{Y_0}Y_\ell)_{\ell\in L\setminus0}$. 
    \end{itemize}
    So the corollary follows from Theorem \ref{theorem-blowup-cofiber-comparison} applied on $(\E_{Y_0}Y_\ell)_{\ell\in L\setminus0}$. 
\end{proof}

\begin{remark}
    Similarly to Corollary \ref{corollary-blowup-cofiber-comparison-minimum}, one can describe the total cofiber of $(\E'_\Psi(i))_{\Psi\subseteq\Pi}$ for any chain $\Psi\subseteq\Pi$ as that of some diagram with final object mapped to $\Bl_{Y_\pi}X$; in fact, more generally, one can describe the total cofiber of $(\E'_\Psi(i))_{\Gamma\subseteq\Psi\subseteq\Pi}$ for any chain $\Gamma\subseteq\Pi$ as that of some diagram with final object mapped to $\E'_\Gamma(i|_\Gamma)$. We omit the details as this remark is not used in the paper. 
\end{remark}

\subsection{Pushout-blowups of excessive squares}
So far, we have finished developing the theory for symmetric monoidality of the Gysin map. For composition, we need a different blowup construction. 

\begin{definition}[pushout-blowup]\label{definition-pushout-blowup}
    For an excessive square of stacks
    $$\begin{tikzcd}
        W\ar[r]\ar[d]&Z\ar[d]\\
        Y\ar[r]&X
    \end{tikzcd}$$
    as in Definition \ref{definition-excessive-diagram}, its \emph{pushout-blowup} is the final stack $\Bl'$ along with
    \begin{itemize}
        \item three divisors $\E'_Y$, $\E'_Z$, and $\E'_W$, and
        \item an excessive map from
        $$\begin{tikzcd}
            (\E'_Y\cap\E'_Z)+\E'_W\ar[r]\ar[d]&\E'_Z+\E'_W\ar[d]\ar[dr]&\\
            \E'_Y+\E'_W\ar[r]\ar[rr,bend right]&\E'_Y+\E'_Z+\E'_W\ar[r]&\Bl'
        \end{tikzcd}$$
        to
        $$\begin{tikzcd}
            W\ar[r]\ar[d]&Z\ar[d]\ar[dr]&\\
            Y\ar[r]\ar[rr,bend right]&Y\sqcup_WZ\ar[r]&X
        \end{tikzcd}$$
    \end{itemize}
    We call $\E'_Y$, $\E'_Z$, and $\E'_W$ the \emph{strict exceptional divisors} over $Y$, $Z$, and $W$, respectively. Here, $\cap$ means taking fiber product over $\Bl'$, and the reason to use $(\E'_Y\cap\E'_Z)+\E'_W$ is that 
    $$(\E'_Y+\E'_W)\sqcup_{(\E'_Y\cap\E'_Z)+\E'_W}(\E'_Z+\E'_W)=\E'_Y+\E'_Z+\E'_W.$$
\end{definition}

\begin{remark}
    The reader might think that the pushout-blowup in Definition \ref{definition-pushout-blowup} is the blowup of the poset
    $$\begin{tikzcd}
        &Z\ar[d]\ar[dr]&\\
        Y\ar[r]\ar[rr,bend right]&Y\sqcup_WZ\ar[r]&X
    \end{tikzcd}$$
    in the sense of Definition \ref{definition-blowup-poset}; however, the author does not think so, at least at this point, although this is probably true in classical algebraic geometry, i.e.\ in the theory of \cite[\S 3.6]{mayeux2025multigradedprojschemes}. 
\end{remark}

\begin{remark}\label{remark-pushout-blowup-cube-characterization}
    It is easy to see that the datum of the excessive map in Definition \ref{definition-pushout-blowup} is equivalent to a map from
    $$\begin{tikzcd}
        \E'_Y\cap\E'_Z\cap\E'_W\ar[rr]\ar[rd]\ar[dd]&&\E'_Z\cap\E'_W\ar[rd]\ar[dd]&\\
        &\E'_Y\cap\E'_Z\ar[rr,crossing over]\ar[dd]&&\E'_Z\ar[dd]\\
        \E'_Y\cap\E'_W\ar[rr]\ar[rd]&&\E'_W\ar[rd]&\\
        &\E'_Y\ar[from=uu,crossing over]\ar[rr]&&\Bl'
    \end{tikzcd}$$
    to
    $$\begin{tikzcd}
        W\ar[rr]\ar[rd]\ar[dd]&&W\ar[rd]\ar[dd]&\\
        &W\ar[rr,crossing over]\ar[dd]&&Z\ar[dd]\\
        W\ar[rr]\ar[rd]&&W\ar[rd]&\\
        &Y\ar[from=uu,crossing over]\ar[rr]&&X
    \end{tikzcd}$$
    that induces surjections $\cI_{Y/X}|_{\Bl'}\twoheadrightarrow\cI_{(\E'_Y+\E'_W)/\Bl'}$, $\cI_{Z/X}|_{\Bl'}\twoheadrightarrow\cI_{(\E'_Z+\E'_W)/\Bl'}$, and $\cI_{(Y\sqcup_WZ)/X}|_{\Bl'}\twoheadrightarrow\cI_{(\E'_Y+\E'_Z+\E'_W)/\Bl'}$, or in other words, surjections on the total fibers (of structure sheaves) in the directions $\{Y,W\}$ (upper and lower faces), $\{Z,W\}$ (left and right faces), and $\{Y,Z,W\}$ (the whole cube). 
\end{remark}

We have to do some preparation before discussing representability of $\Bl'$. 

\begin{proposition}\label{proposition-blowup-edges-excessive-preserve-immersion}
    Let
    $$
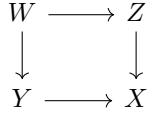
\begin{tikzcd}
        W\ar[r]\ar[d]&Z\ar[d]\\
        Y\ar[r]&X
    \end{tikzcd}$$
    be an excessive square. Then the canonical map $\Bl_WZ\to\Bl_YX$ is a closed immersion, and $\N_{\Bl_WZ/\Bl_YX}=\fib(\N_{Z/X}|_{\Bl_WZ}\to\N_{W/Y}|_{\E_WZ})(\E_WZ)$. Furthermore, if all four maps above are quasismooth closed immersions, then so is $\Bl_WZ\to\Bl_YX$. 
\end{proposition}

\begin{proof}
    We may assume that all four schemes are affine, so we have ring surjections
    $$
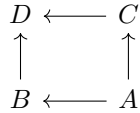
\begin{tikzcd}
        D&C\ar[l]\\
        B\ar[u]&A\ar[l]\ar[u]
    \end{tikzcd}$$
    and we set $I=\fib(A\to B)$ and $J=\fib(C\to D)$. In this case the excessivity condition is saying that $I\twoheadrightarrow J$, so the map of the Rees graded rings $I^\bullet\to J^\bullet$ is surjective, and its $\Proj$ is a closed immersion, showing that $\Bl_WZ\to\Bl_YX$ is a closed immersion by \cite[Theorem 4.7]{hekkingkhanrydh2025deformationnormalbundleblowups}. To show the equality concerning normal bundles, we use a Yoneda argument. Let $R$ be a ring, $M$ an $R$-module, and
    $$
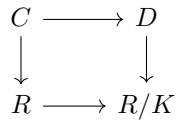
\begin{tikzcd}
        C\ar[d]\ar[r]&D\ar[d]\\
        R\ar[r]&R/K
    \end{tikzcd}$$
    an $R$-point of $\Bl_WZ$, so $J\twoheadrightarrow K$. Expanding the definitions, we can see that a map $\N_{\Bl_WZ/\Bl_YX}|_R\to M$ amounts to the datum of the following diagram
    $$\begin{tikzcd}
        A\ar[rr]\ar[rd]\ar[dd]&&B\ar[rd]\ar[dd]&\\
        &C\ar[rr,crossing over]&&D\ar[dd]\\
        R\oplus M[1]\ar[rr]\ar[rd]&&R/K\oplus M/KM[1]\ar[rd]&\\
        &R\ar[rr]\ar[from=uu,crossing over]&&R/K
    \end{tikzcd}$$
    where the direct sums denote trivial square-zero extensions. Note that this datum also amounts to a commutative diagram
    $$\begin{tikzcd}
        \N_{Z/X}|_R\ar[r]\ar[d]&\N_{W/Y}|_{R/K}\ar[d]\\
        M\ar[r]&M/KM
    \end{tikzcd}$$
    where the right vertical arrow is $R/K$-linear and the others are $R$-linear. Now the claim follows from Lemma \ref{lemma-invertible-fiber-yoneda} below. 

    Note that if $\N_{Z/X}$ and $\N_{W/Y}$ are locally free, since $\N_{Z/X}|_W\twoheadrightarrow\N_{W/Y}$, it is easy to see from what we just proved that $\N_{\Bl_WZ/\Bl_YX}$ is locally free. Therefore, to justify the last sentence, it remains to see that under its assumption, $\Bl_WZ\to\Bl_YX$ is a pseudocoherent map. This follows from the fact that $\Bl_YX\to X$, $\Bl_WZ\to Z$, and $Z\to X$ are all easily seen to be pseudocoherent. 
\end{proof}

\begin{lemma}\label{lemma-invertible-fiber-yoneda}
    Let $R$ be a ring and $K\to R$ be an invertible ideal. Let $N$ be an $R$-module, $Q$ be an $R/K$-module, and let $p\colon N\to Q$ be a surjection as $R$-modules. Then for any $R$-module $M$, the datum of a commutative diagram
    $$\begin{tikzcd}
        N\ar[r,"p"]\ar[d]&Q\ar[d]\\
        M\ar[r]&M/KM
    \end{tikzcd}$$
    where the right vertical arrow is $R/K$-linear and the others are $R$-linear, amounts to a map $\fib(p)\to KM$ through taking fibers.
\end{lemma}

\begin{proof}
    The lemma is about mapping animas in the category of such $(N,Q,p)$. Since $(N,Q,p)\mapsto\fib
    (p)$ preserves colimits, it suffices to prove for a set of generators in the category, which is easy: note that $(R,R/K)$ and $(R,0)$ generate the category, and it is straightforward to verify the claim in these two cases. 
\end{proof}

\begin{remark}
    The closed immersion $\Bl_WZ\to\Bl_YX$ is the derived version
    of what is classically called the \emph{strict transform} of $Z$ in $\Bl_YX$. Note that in the derived setting the strict transform depends on $W$. 
\end{remark}

Fix a motivic setup $\cP$ as in Definition \ref{definition-motivic-setup} and a base scheme $S$. 

\begin{theorem}\label{pushout-blowup-description}
    For any excessive square in $\cP_S$
    $$\begin{tikzcd}
        W\ar[r]\ar[d]&Z\ar[d]\\
        Y\ar[r]&X
    \end{tikzcd}$$
    where all the arrows are quasismooth, its pushout-blowup $\Bl'$ is representable in $\cP_S$, and so is the fiber product over $\Bl'$ of every subset of $\{\E'_Y,\E'_Z,\E'_W\}$. They can be described as follow: 
    \begin{enumerate}
        \item $\Bl'=\Bl_{\Bl_{\E_WY}(\E_ZX)_W}\Bl_{\Bl_WY}\Bl_ZX=\Bl_{\Bl_{\E_WZ}(\E_YX)_W}\Bl_{\Bl_WZ}\Bl_YX$.
        \item $\E'_Y=\Bl_{\E_{\E_WY}(\E_ZX)_W}\E_{\Bl_WY}\Bl_ZX=\Bl_{\Bl_{\E_WZ}(\E_YX)_W}\Bl_{\E_WZ}\E_YX$.
        \item $\E'_Z=\Bl_{\Bl_{\E_WY}(\E_ZX)_W}\Bl_{\E_WY}\E_ZX=\Bl_{\E_{\E_WZ}(\E_YX)_W}\E_{\Bl_WZ}\Bl_YX$. 
        \item $\E'_Y\cap\E'_Z=\Bl_{\E_{\E_WY}(\E_ZX)_W}\E_{\E_WY}\E_ZX=\Bl_{\E_{\E_WZ}(\E_YX)_W}\E_{\E_WZ}\E_YX$. 
        \item $\E'_W=\E_{\Bl_{\E_WY}(\E_ZX)_W}\Bl_{\Bl_WY}\Bl_ZX=\E_{\Bl_{\E_WZ}(\E_YX)_W}\Bl_{\Bl_WZ}\Bl_YX$
        \item $\E'_Y\cap\E'_W=\E_{\E_{\E_WY}(\E_ZX)_W}\E_{\Bl_WY}\Bl_ZX=\E_{\Bl_{\E_WZ}(\E_YX)_W}\Bl_{\E_WZ}\E_YX$.
        \item $\E'_Z\cap\E'_W=\E_{\Bl_{\E_WY}(\E_ZX)_W}\Bl_{\E_WY}\E_ZX=\E_{\E_{\E_WZ}(\E_YX)_W}\E_{\Bl_WZ}\Bl_YX$.
        \item $\E'_Y\cap\E'_Z\cap\E'_W=\E_{\E_{\E_WY}(\E_ZX)_W}\E_{\E_WY}\E_ZX=\E_{\E_{\E_WZ}(\E_YX)_W}\E_{\E_WZ}\E_YX$. 
    \end{enumerate}
\end{theorem}

\begin{proof}
    We use the equivalent formulation in Remark \ref{remark-pushout-blowup-cube-characterization}. Using the universal property of a single blowup, we see that the map there factors through
    $$\begin{tikzcd}
        \E_WY\ar[rr]\ar[rd]\ar[dd]&&(\E_ZX)_W\ar[rd]\ar[dd]&\\
        &\E_WY\ar[rr,crossing over]\ar[dd]&&\E_ZX\ar[dd]\\
        \E_WY\ar[rr]\ar[rd]&&(\E_ZX)_W\ar[rd]&\\
        &\Bl_WY\ar[from=uu,crossing over]\ar[rr]&&\Bl_ZX
    \end{tikzcd}$$
    inducing surjections on the total fibers of upper and lower faces, left and right faces, and the whole cube. The surjectivity condition on the whole cube is
    $$(\cI_{\Bl_WY/\Bl_ZX}\otimes\cI_{\E_ZX/\Bl_ZX})|_{\Bl'}\twoheadrightarrow\cI_{\E'_Y/\Bl'}\otimes\cI_{(\E'_Z+\E'_W)/\Bl'};$$
    by the surjectivity condition on the left and right faces, one can easily see that $\cI_{\E_ZX/\Bl_ZX}|_{\Bl'}=\cI_{(\E'_Z+\E'_W)/\Bl'}$ and is invertible, so the surjectivity condition on the whole cube implies the surjectivity condition in the direction $\{Y\}$, i.e.\ the left-to-right direction. Therefore, the map further factors through
    $$\begin{tikzcd}
        \E_{\E_WY}(\E_ZX)_W\ar[rr]\ar[rd]\ar[dd]&&\Bl_{\E_WY}(\E_ZX)_W\ar[rd]\ar[dd]&\\
        &\E_{\E_WY}\E_ZX\ar[rr,crossing over]\ar[dd]&&\Bl_{\E_WY}\E_ZX\ar[dd]\\
        \E_{\E_WY}(\E_ZX)_W\ar[rr]\ar[rd]&&\Bl_{\E_WY}(\E_ZX)_W\ar[rd]&\\
        &\E_{\Bl_WY}\Bl_ZX\ar[from=uu,crossing over]\ar[rr]&&\Bl_{\Bl_WY}\Bl_ZX
    \end{tikzcd}$$
    inducing surjections on the total fibers of left-to-right arrows, upper and lower faces, left and right faces, and the whole cube. By the same reason, since the left-to-right arrows are all divisors, the surjectivity condition on the total fibers of upper and lower faces implies the surjectivity condition in the direction $\{W\}$, i.e.\ the back-to-front direction. Therefore, the map further factors through
    $$\begin{tikzcd}[column sep=-5.4em]
        \E_{\E_{\E_WY}(\E_ZX)_W}\E_{\E_WY}\E_ZX\ar[rr]\ar[rd]\ar[dd]&&\E_{\Bl_{\E_WY}(\E_ZX)_W}\Bl_{\E_WY}\E_ZX\ar[rd]\ar[dd]&\\
        &\Bl_{\E_{\E_WY}(\E_ZX)_W}\E_{\E_WY}\E_ZX\ar[rr,crossing over]\ar[dd]&&\Bl_{\Bl_{\E_WY}(\E_ZX)_W}\Bl_{\E_WY}\E_ZX\ar[dd]\\
        \E_{\E_{\E_WY}(\E_ZX)_W}\E_{\Bl_WY}\Bl_ZX\ar[rr]\ar[rd]&&\E_{\Bl_{\E_WY}(\E_ZX)_W}\Bl_{\Bl_WY}\Bl_ZX\ar[rd]&\\
        &\Bl_{\E_{\E_WY}(\E_ZX)_W}\E_{\Bl_WY}\Bl_ZX\ar[from=uu,crossing over]\ar[rr]&&\Bl_{\Bl_{\E_WY}(\E_ZX)_W}\Bl_{\Bl_WY}\Bl_ZX
    \end{tikzcd}$$
    which is already a diagram of three divisors intersecting in an ambient stack, so:
    \begin{enumerate}
        \item $\Bl'=\Bl_{\Bl_{\E_WY}(\E_ZX)_W}\Bl_{\Bl_WY}\Bl_ZX$.
        \item $\E'_Y=\Bl_{\E_{\E_WY}(\E_ZX)_W}\E_{\Bl_WY}\Bl_ZX$.
        \item $\E'_Z=\Bl_{\Bl_{\E_WY}(\E_ZX)_W}\Bl_{\E_WY}\E_ZX$.
        \item $\E'_Y\cap\E'_Z=\Bl_{\E_{\E_WY}(\E_ZX)_W}\E_{\E_WY}\E_ZX$.
        \item $\E'_W=\E_{\Bl_{\E_WY}(\E_ZX)_W}\Bl_{\Bl_WY}\Bl_ZX$.
        \item $\E'_Y\cap\E'_W=\E_{\E_{\E_WY}(\E_ZX)_W}\E_{\Bl_WY}\Bl_ZX$.
        \item $\E'_Z\cap\E'_W=\E_{\Bl_{\E_WY}(\E_ZX)_W}\Bl_{\E_WY}\E_ZX$.
        \item $\E'_Y\cap\E'_Z\cap\E'_W=\E_{\E_{\E_WY}(\E_ZX)_W}\E_{\E_WY}\E_ZX$. 
    \end{enumerate}
    That they all remain in $\cP_S$ follows from Proposition \ref{proposition-blowup-edges-excessive-preserve-immersion}. The other half of the theorem follows by the symmetry of $Y$ and $Z$. 
\end{proof}

\subsection{Possible directions of further generalization}

\begin{remark}[more functoriality]
    The blowups introduced in this section should have more functoriality than in excessive maps, just as the $\Proj$ construction of graded rings has more functoriality than in certain graded maps: it is also functorial in certain maps that scale the grading. We plan to pursue such functoriality in future work, as we may need to use it in comparing the $\MS$ of \cite{annala-hoyois-iwasa-2023} and the $\logSH$ of \cite{binda-park-ostvaer-logsh}. It will probably involve the $\Proj$ construction of $\NN^n$-graded rings, i.e.\ an animated version of the theory in \cite{brenner-schroer-multi-proj} and \cite{mayeux2025multigradedprojschemes}. 
\end{remark}

\begin{remark}[pushout-blowup for cubes]
    One can define pushout-blowups for excessive $[1]^n$-diagrams similar to Definition \ref{definition-pushout-blowup}, with descriptions similar to Theorem \ref{pushout-blowup-description}, consisting of $2^n-1$ steps of blowups. This will be needed for composing $n$ Gysin maps, generalizing Theorem \ref{theorem-composition-gysin} that composes two Gysin maps. Ideally, the author expects a common generalization of such pushout-blowups and poset blowups as in Definition \ref{definition-blowup-poset}, but he has no concrete idea on this yet. 
\end{remark}

\printbibliography

\end{document}